\RequirePackage{fix-cm}

\documentclass[smallcondensed,stropt,envcountsect,envcountsame]{svjour3}

\smartqed 
\usepackage{setspace}
\usepackage{color}
\usepackage{longtable}
\usepackage{enumitem}
\usepackage{multirow}
\usepackage{url}
\usepackage[misc]{ifsym}
\usepackage{float}

\usepackage{amssymb,amsfonts,amsmath,amsthm}
\usepackage{adjustbox,graphicx}
\usepackage{epstopdf}
\usepackage{epsfig}

\newtheorem{thm}{\bf{Theorem}}[section]
\newtheorem{lem}[thm]{\bf{Lemma}}
\newtheorem{sub}[thm]{\bf{Subroutine}}
\newtheorem{df}[thm]{\bf{Definition}}
\newtheorem{cor}[thm]{\bf{Corollary}}
\newtheorem{rem}[thm]{\bf{Remark}}
\newtheorem{ass}[thm]{\emph{Assumption}}

\newcommand{\dom}{\operatorname{dom}}
\newcommand{\intt}{\operatorname{int}}
\newcommand{\cl}{\operatorname{cl}}

\newcommand{\rank}{\operatorname{rank}}
\newcommand{\Id}{\operatorname{\tt I}}

\newcommand{\dist}{\operatorname{dist}}
\newcommand{\conv}{\operatornamewithlimits{conv}}
\newcommand{\argmin}{\operatornamewithlimits{argmin}}

\newcommand{\ri}{\operatorname{ri}}
\newcommand{\Span}{\operatorname{span}}

\newcommand{\R}{\operatorname{\mathbb{R}}}

\newcommand{\C}{\operatorname{\mathcal{B}}}
\newcommand{\U}{\operatorname{\mathcal{U}}}
\newcommand{\V}{\operatorname{\mathcal{V}}}
\newcommand{\nul}{\operatorname{null}}

\newcommand{\VU}{\operatorname{\mathcal{VU}}}
\newcommand{\OU}{\operatorname{{U}}}
\newcommand{\OV}{\operatorname{{V}}}
\newcommand{\Prox}{\operatorname{Prox}}
\newcommand{\Proj}{\operatorname{Proj}}

\newcommand{\myinput}[4]{%
\begin{center}
\scriptsize\input{#1.txt}\end{tabular} \caption{23}\label{tab-#3} \end{table}\end{center}}

\newcommand{\Va}{{\sc CompBun}}
\newcommand{\Vb}{{\sc ExBun}}
\newcommand{\Vc}{{\sc InexBun}}
\newcommand{\Vd}{{\sc DFO-VU}}
\newcommand{\Ve}{{\sc Nomad}}
\newcommand{\ttRA}{{{RA}}}

\newcommand{\lttRA}{{{RA}}}

\newcommand{\maxs}{$\max_{\tt S}$}

\begin{document}

\title{A derivative-free $\VU$-algorithm for convex finite-max problems}
\author{Warren Hare\and Chayne Planiden\and Claudia Sagastiz\'abal}
\institute{Warren Hare \at Department of Mathematics, University of British Columbia, Okanagan Campus\\Kelowna, B.C. V1V 1V7, Canada\\Research partially supported by Natural Sciences and Engineering Research Council (NSERC) of Canada Discovery Grant 355571-2013.\\warren.hare@ubc.ca, ORCID 0000-0002-4240-3903
 \and 
Chayne Planiden \Letter\at School of Mathematics and Applied Statistics, University of Wollongong\\Wollongong, NSW, 2500, Australia\\Research supported by UBC UGF and by NSERC of Canada.\\chayne@uow.edu.au, ORCID 0000-0002-0412-8445, Tel: +61 (02) 4221 4564.\and Claudia Sagastiz\'abal\at 
Adjunct Researcher IMECC - UNICAMP, 13083-859, Campinas, SP, Brazil. \\
Research partially supported by CNPq Grant 303905/2015-8, CEMEAI, and FAPERJ.\\ sagastiz@unicamp.br, ORCID 0000-0002-9363-9297}
\date{Received: date / Accepted: date}
\maketitle
\begin{abstract}
The $\VU$-algorithm is a superlinearly convergent method for minimizing nonsmooth, convex functions. At each iteration, the algorithm works
 with a certain $\V$-space and its orthogonal $\U$-space, such that the nonsmoothness of the objective function is concentrated on its projection onto the $\V$-space, and on the $\U$-space the projection is smooth. This structure allows for an alternation between a Newton-like step where the function is smooth, and a proximal-point step that is used to find iterates with promising $\VU$-decompositions.  We establish a derivative-free variant of the $\VU$-algorithm for convex finite-max objective functions.  We show global convergence and provide numerical results from a proof-of-concept implementation, which demonstrates the feasibility and practical value of the approach. We also carry out some tests using nonconvex functions and discuss the results.
\end{abstract}

\section{Introduction}\label{sec:intro}

In this paper, we consider the finite-dimensional, unconstrained minimization problem
\begin{equation}\label{prob1}
\min\limits_{x\in\R^n}f(x)
\end{equation}
with $f$ a nonsmooth, proper, \emph{convex finite-max function},
    $$f(x) = \max_{i=1, ..., m} f_i(x),$$where for $i\in\{1,2,\ldots,m\}$ the subfunctions
$f_i:\R^n\to\R$ are of class $\mathcal{C}^{2+}$. We work under the assumption that there is an oracle delivering only function values. We refer to the oracle as a ``grey-box'' because, for a given $x\in\R^n$ the oracle informs not only $f(x)$, but also the values of each subfunction $f_i(x)$.  Our goal is to exploit the grey-box information in a derivative-free optimization setting, blended with a suitable variant of bundle methods. \par

We remark that this finite-max grey-box structure is a natural outcome of any optimization problem where multiple simulations are employed and then a worst-case outcome must be optimized.  For example, see \cite{BigdeliHareNutiniTesfamariam-2016}, where multiple simulations were used to analyze structural quality of buildings after seismic retrofitting was applied.  In that case, the goal was to minimize the worst-case damage. 

While we acknowledge that convexity of the objective in a finite-max grey-box function would be difficult to confirm, we note that many nonsmooth optimization problems take a convex finite-max structure (e.g., maximize eigenvalue optimization).  Regardless, we view this as a necessary starting point for research in this direction.  If an algorithm cannot be designed and proven to converge for convex finite-max grey-box functions, then designing an algorithm for a more general setting would be intractable. Although the convergence results in this paper are proved for the convex case only, our numerical testing includes a set of nonconvex trials as well. Those results are discussed further in Section \ref{sec:numerics}.

Derivative-free optimization (DFO) studies algorithms that use only function values to minimize the objective \cite{AudetHare_DFObook,connscheinbergvicente_book}.  Due to its broad applicability, particularly for optimizing simulations, DFO has seen many successful applications in the past decade (see \cite{Audet2014,HareNutiniTesfamariam-2013} and references therein).  DFO algorithms can be split broadly into two categories, direct search methods and model-based methods.  Model-based methods approximate the objective function with a model function, and then use the model to guide the optimization algorithm \cite[Part 4]{AudetHare_DFObook}.

While model-based methods were originally designed for optimization of smooth objective functions (see, for example, \cite{ConnScheinbergToint1997,CustVicente2007,HareLucet2014,Powell2008}), recent research has moved away from this assumption \cite{Hare-Macklem-2013,Hare-Nutini-2013,LarsonMenickellyWild2016}. In \cite{Hare-Macklem-2013,Hare-Nutini-2013}, it is assumed that the true objective function takes the form $$F = \max\{ f_i : i = 1, 2, \ldots, nf \},$$ where each $f_i$ is given by a blackbox that provides only function values.  In \cite{LarsonMenickellyWild2016}, it is assumed that the objective function takes the form $F = \sum_{i=1}^m |f_i|$, where each $f_i$ is given by a blackbox that provides only function values.  In each case, it is shown that such information allows for the creation of a convergent model-based DFO algorithm for nonsmooth optimization.

Bundle methods proceed by collecting information (function values and subgradient vectors) along iterations, then using that information to build a model of the objective function and seek a new incumbent solution (often called a serious point in bundle method literature) \cite{bonnans-gilbert-lemarechal-sagastizabal-2006,deOliveiraSagaLemar2014}.  Bundle methods have been widely established as the most robust and effective technique for nonsmooth optimization \cite{ApkarianNollRavanbod2016,bonnans-gilbert-lemarechal-sagastizabal-2006,HareSagaSolodov2016,KarmitsaBagirovTaheri2017,Kiwiel2006,NollProtRondepierre2008,deOliveiraSolodov2016}. They are also well-known for their ability to work with the structure of a given problem.  Specialized bundle methods have been developed considering eigenvalue optimization \cite{HelmbergOvertonRendl2014,HelmbergRendl2000}, sum functions \cite{proxsplit,FrangioniGorgone2014}, chance-constrained problems \cite{vanackooij2014level}, composite functions \cite{LewisWright2016,sagastizabal-2013} and difference-convex functions \cite{FuduliGaudiosoGial2004,JokiBagirovKarmitsa2017,deOliveira2017}.

Of particular interest to this paper is the $\mathcal{VU}$-algorithm for convex minimization \cite{vualg}.  The $\mathcal{VU}$-algorithm alternates between a proximal-point step and a `$\mathcal{U}$-Newton' step (see Step 4 of the Conceptual $\VU$-algorithm in Subsection \ref{ssec:conceptualVU}) to achieve superlinear convergence in the minimization of nonsmooth convex functions \cite{vualg}.
The $\VU$-algorithm  has proven effective in dealing with the challenges that arise in the minimization of nonsmooth convex functions \cite{numvu,vudecomp,vutheory,vusmoothness,vualg}. It continues to be a method of interest in the optimization community, having been expanded to use on convex functions with primal-dual gradient structure \cite{mifflin2000functions,vutheory,mifflin2003primal} and even some nonconvex functions \cite{vusmoothness}. The basic tenet is to separate the space into two orthogonal subspaces, called the $\V$-space and the $\U$-space, such that near the current iteration point the nonsmoothness of $f$ is captured in the $\V$-space and the smoothness of $f$ is captured in the $\U$-space.  This procedure is known as $\VU$-decomposition.  Once this is accomplished, one takes a proximal-point step ($\V$-step) parallel to the $\V$-space, in order to find incumbent solutions with favourable $\VU$-decompositions, then a Newton-like step ($\U$-step) parallel to the $\U$-space. This process is repeated iteratively and converges to a minimizer of $f.$ In fact, the $\VU$-algorithm has been proved superlinearly convergent under reasonable conditions \cite{vualg}. Further details on the $\VU$-algorithm can be found in Section \ref{sec:backgroundandVU} of the present, and proof of convergence 
(for oracles delivering subgradient information)
is given in \cite{vualg}. Techniques used in the implementation of the $\VU$-algorithm are also currently
being used in gradient sampling methods \cite{salomao2016differentiability,salomao2017local,salomao2016second}, see Section \ref{sec:conc} for details.\par In order to apply the $\VU$-algorithm, at each iteration it is necessary to do the $\VU$-decomposition, compute the proximal point to apply the $\V$-step, then compute the $\U$-gradient and $\U$-Hessian to apply the $\U$-step (each of these computations is formally defined in Section \ref{sec:backgroundandVU}). In our grey-box optimization setting, none of these objects is directly available. However, in \cite{numvu} it was shown that the $\VU$-decomposition, $\U$-gradient, and $\U$-Hessian can be approximated numerically with controlled precision for finite-max functions. Moreover, in \cite{hareplaniden2016} a derivative-free algorithm for computing proximal points of convex functions that only requires approximate subgradients was developed.  Finally, in \cite{Hare-Nutini-2013} it was shown how to approximate subgradients for convex finite-max functions using only function values.  Combined, these three papers provide a sufficient foundation to develop a derivative-free $\VU$-algorithm suitable for our grey-box optimization setting. We show that at each iteration, one can approximate subgradients of the objective function as closely as one wishes and use the inexact first-order information to obtain approximations of all the necessary components of the algorithm. We prove that the results of global convergence in \cite{vualg} can be extended to the framework of inexact gradients and Hessians.\par 

The remainder of this paper is organized as follows. We finish the present section with notation and a statement of our assumptions on the objective function. Section \ref{sec:backgroundandVU} contains the basic definitions used in this paper and provides a brief primer on the $\VU$-algorithm. Section \ref{sec:approx} presents details on the simplex gradient and Frobenius norm, which are tools needed for the DFO version of the algorithm, and establishes the DFO $\VU$-algorithm.  Section \ref{sec:approx} includes the DFO $\VU$-algorithm pseudo-code and provides some comments comparing our algorithm to other established DFO methods. In Section \ref{sec:conv}, we examine the convergence properties of the algorithm. In Section \ref{sec:numerics}, we showcase numerical results obtained for randomly generated max-of-quadratic functions. The numerical behaviour of the method on nonconvex functions is also explored, resulting in insight on its good performance (not yet backed up by a convergence analysis). Section \ref{sec:conc} summarizes this work and discusses future possibilities of this field of research, in particular regarding recent variants of gradient sampling methods \cite{salomao2016differentiability,salomao2017local,salomao2016second}.

\subsection{Notation}

We work in the finite-dimensional space $\R^n$, with inner product $x^\top y=\sum_{i=1}^nx_iy_i$ and induced norm $\|x\|=\sqrt{x^\top x}$. We use standard notation and concepts from convex analysis found in \cite{rockwets}. The identity matrix is denoted by $\Id$. We denote by $B_\delta$ the open ball of radius $\delta$ about the origin. Given a set $S$, we denote its interior, closure and relative interior by $\intt(S)$, $\cl(S)$ and $\ri(S)$, respectively. We denote the smallest convex set containing $S$, i.e. the convex hull of $S$, by $\conv S$. The span of a set of vectors $T$, denoted by $\Span T$, is the set of all linear combinations of the vectors in $T$.

As the objective function $f$ is convex and finite-valued, the \emph{subdifferential} of $f$ at a point $\bar{x}$, defined by
the set
$$\partial f(\bar{x}) = \{g\in\R^n: f(x) \geq f(\bar{x}) + g^\top(x- \bar{x})~\mbox{for all}~x\in\R^n\},$$
is well-defined and never empty. An element $g\in\partial f(\bar{x})$ is called a \emph{subgradient} of $f$ at $\bar{x}.$ The \emph{$\varepsilon$-subdifferential} of $f$ at $x$ is denoted $\partial_\varepsilon f(x)$ (with $g\in\partial_\varepsilon f(x)$ called an \emph{$\varepsilon$-subgradient}) and is defined by
$$\partial_\varepsilon f(\bar{x})=\{g\in\R^n:f(x)\geq f(\bar{x})+g^\top(x-\bar{x})-\varepsilon\mbox{ for all }x\in\R^n\}.$$
Given a finite-max function, the active indices provide an alternate manner of constructing the subdifferential:
\begin{equation}\label{convexhull}
\partial f(\bar{x}) = \conv\{\nabla f_i(\bar{x}) : i \in A(\bar{x})\},\end{equation}
where $A(\bar{x}) = \{ i : f_i(\bar{x}) = f(\bar{x})\}$.
\begin{df}
A function $f:\R^n\to\R^m$ is a \emph{$\mathcal{C}^k$ ($\mathcal{C}^{k+}$) function} if all partial derivatives of $f$ of degree $0$ to $k$ exist and are (locally Lipschitz) continuous.
\end{df}
\begin{df}
Given a differentiable function $f:\R^n\to\R^m$, the \emph{Jacobian} of $f$, written $J_f$, is the matrix of all partial derivatives of $f$:
$$J_f=\left[\begin{array}{c c c}\frac{\partial f}{\partial x_1}&\cdots&\frac{\partial f}{\partial x_n}\end{array}\right]=\left[\begin{array}{c c c}
\frac{\partial f_1}{\partial x_1}&\cdots&\frac{\partial f_1}{\partial x_n}\\
\vdots&\ddots&\vdots\\
\frac{\partial f_m}{\partial x_1}&\cdots&\frac{\partial f_m}{\partial x_n}\end{array}\right].$$
\end{df}
\begin{df}
Given a point $\bar{x}$ and a {\em proximal parameter} $r>0,$ the \emph{proximal mapping}, denoted $\Prox_f^r(\bar{x})$, is defined by
    $$\Prox_f^r(\bar{x})=\argmin\limits_y\left\{f(y)+\frac{r}{2}\|y-\bar{x}\|^2\right\}.$$
\end{df}

\subsection{Assumptions}\label{ssec:assumptions}

Throughout this paper, we assume the following for Problem \eqref{prob1}.
\begin{ass}\label{ass:finitemax}
The objective function $f:\R^n\to\R$ is convex and defined through the maximum of a finite number of subfunctions,
    $$f(x)=\max_{i=1, ..., m} f_i(x),$$
where each $f_i \in \mathcal{C}^{2+}.$ Furthermore, at each given point $\bar{x}$ the grey-box returns the individual function values $f_i(x)$ and as such also provides indices of active subfunctions, i.e., 
	$$A(\bar{x})=\{i:f_i(\bar{x})=f(\bar{x})\}.$$
\end{ass}
\begin{ass}\label{ass:compactlowerlevel}
The objective function $f$ has compact {\em lower level sets}, that is, the set
$$S_\beta=\{x:f(x)\leq\beta\}$$
is compact for any choice of $\beta\in\R.$
\end{ass}
\begin{ass}\label{ass:affineindep}
For any fixed $\bar{x}\in\R^n,$ the set of {\em active gradients} $$\{\nabla f_i(\bar{x}) : i \in A(\bar{x}) \}$$
is affinely independent. That is, the only scalars $\lambda_i$ that satisfy
    $$\sum\limits_{i\in A(\bar{x})}\lambda_i\nabla f_i(\bar{x})=0,~\sum\limits_{i\in A(\bar{x})}\lambda_i=0$$
are $\lambda_i=0$ for all $i\in A(\bar{x}).$
\end{ass}

\section{Background and $\VU$-theory}\label{sec:backgroundandVU}

At any point $x\in\R^n$, the space can be split into two orthogonal subspaces called the $\U$-space and the $\V$-space, such that the nonsmoothness of $f$ is captured entirely in the $\V$-space, while on the $\U$-space $f$ behaves smoothly. 
The $\VU$-method tracks a smooth trajectory of $f$ along which
a Newton-like update can be done, even though the function is not differentiable everywhere. 
The smooth trajectory is special in the sense that its $\VU$-decomposition
has a $\V$-component that converges faster than its $\U$-component. 
Along the smooth trajectory, the rate of
convergence is driven by the speed of the $\U$-component, which is updated using a (fast) Newton step.
This explains the superlinear speed of convergence  of conceptual
$\VU$-methods under certain assumptions (such
as having perfect knowledge of the full subdifferential of $f$ at a minimizer and of the
matrices involved in a second-order expansion of the smooth trajectory); see Section \ref{ssec:conceptualVU}.

The algorithmic identification of the smooth trajectories is possible thanks to two useful relations established in \cite{correaconvergence1993} and \cite{mifflin-sagastizabal-2002}.
Specifically, the first work shows that a bundle mechanism gives asymptotically the exact value of the proximal point operator at a given point, for oracles delivering subgradient information. The second work, see Theorem~\ref{thm:pri}, relates proximal points with the smooth trajectory. A sound combination of these elements gives an implementable form to the conceptual $\VU$-algorithm in Section~\ref{ssec:conceptualVU}. Our contribution extends the relations above to the grey-box oracle in Assumption~\ref{ass:finitemax}, by
suitably coupling those bundle-method results with DFO techniques to derive the implementable DFO $\VU$-algorithm in Section~\ref{DFOVU-alg}.

The main relations and formal definitions of the $\VU$-decomposition, the $\U$-Lagrangian that yields the smooth trajectories, and the proximal point mapping are recalled below.

\begin{df}
Fix $\bar{x}\in\R^n$ and let $g\in\ri(\partial f(\bar{x})).$ The \emph{$\VU$-decomposition} of $\R^n$ for $f$ at $\bar{x}$ is the separation of $\R^n$ into the following two orthogonal subspaces:
    $$\V(\bar{x})=\Span\{\partial f(\bar{x})-g\}\qquad\mbox{ and }\qquad U(\bar{x})=[\V(\bar{x})]^\perp.$$
\end{df}
\noindent This decomposition is independent of the choice of $g\in\ri\partial f(\bar{x})$ \cite[Proposition 2.2]{ulagconv}. With $\OV\in\R^{n\times\dim\V}$ a basis matrix for the $\V$-space and $\OU\in\R^{n\times\dim\U}$ an orthonormal basis matrix for the $\U$-space, every $x\in\R^n$ can be decomposed into components $x_{\U}\in\R^{\dim\U}$ and $x_{\V}\in\R^{\dim\V}$ \cite[Section 2]{ulagconv}. Defining
    $$x_{\U}=\OU^\top x\mbox{ \qquad and\qquad } x_{\V}=\left(\OV^\top\OV\right)^{-1}\OV^\top x,$$
we write
    $$x=
\OU x_{\U}+\OV x_{\V}.$$
Henceforth, we use the notation $\R^{|\U|}$ and $\R^{|\V|}$ to represent $\R^{\dim\U}$ and $\R^{\dim\V},$ respectively.

Given a subgradient $g\in\partial f(\bar{x})$ with $\V$-component $g_{\V},$ the \emph{$\U$-Lagrangian} of $f,$ $L_{\U}(u;g_{\V}):\R^{|\U|}\rightarrow\R$ is defined \cite[Section 2.1]{vualg} as follows:
    $$L_{\U}(u;g_{\V})=\min_{v\in\R^{|\V|}}\left\{f\left(\bar{x}+\OU u+\OV v\right)-g^\top\OV v\right\}.$$
The associated set of $\V$-space minimizers is
    $$\begin{array}{rcl}
		W(u;g_{\V})&=&\argmin_{v\in\R^{|\V|}}\left\{f\left(\bar{x}+\OU u+\OV v\right)-g^\top\OV v\right\}\\
		&=&\{\OV v:L_{\U}(u;g_{\V})=f(\bar{x}+\OU u+\OV v)-g^\top\OV v\}.
		\end{array}$$
The $\U$-gradient $\nabla L_{\U}(0;g_{\V})$ and the $\U$-Hessian $\nabla^2L_{\U}(0;g_{\V})$ are then defined as the gradient and Hessian, respectively, of the $\U$-Lagrangian. For $f$ convex, each $\U$-Lagrangian is a convex function that is differentiable at $u=0,$ with
    $$\nabla L_{\U}(0;g_{\V})=g_{\U}=\OU^\top g=\OU^\top\tilde{g}\mbox{ for all }\tilde{g}\in\partial f(x)\qquad\mbox{\cite[Section 2.1]{vualg}.}$$  If $L_{\U}(u;g_{\V})$ has a Hessian at $u=0,$ then the second-order expansion of $L_{\U}$ also provides a second-order expansion of $f$ in the $\U$-space, which thereby allows for a so-called $\U$-Newton step.  General conditions for existence of the $\U$-Hessian are found in \cite{vualg}.  However, for the purpose of this paper, we note that Assumptions \ref{ass:finitemax} and \ref{ass:affineindep}, combined with $g\in\ri\partial f(x),$ imply the existence of a $\U$-Hessian at the origin \cite[Lemma 2.6, Lemma 3.5]{numvu}. When $\bar{x}$ minimizes $f,$ we have $0\in\partial f(\bar{x})$ \cite[Section 2.1]{vualg}. This gives $$\nabla L_{\U}(0;g_{\V})=0\mbox{ for all }g\in\partial f(\bar{x}),$$ and $L_{\U}$ is minimized at $u=0,$ which subsequently yields $L_{\U}(0;0)=f(\bar{x}).$

\subsection{Relation with the proximal point operator}

The second-order expansion of $f$ in the $\U$-space allows the $\VU$-algorithm to take $\U$-Newton steps, which in turn allows for rapid convergence.  However, in order to be effective, the algorithm must seek out iterates where the $\U$-space at the iterate lines up with the $\U$-space at the minimizer.  This is accomplished through the proximal point operation. When $f$ is convex, the proximal mapping $\Prox_f^r$ is a singleton, called the proximal point. When computed close to
a minimizer $\bar x$, the proximal point has a very important relationship with the smooth
trajectory provided by the $\U$-Lagrangian minimizers, called primal track in
\cite[\S1]{vusmoothness}. 

As shown in \cite[Corollary 3.5]{ulagconv}, for sufficiently small $u$ the trajectories created
from the set of $\V$-space minimizers, that is \(\bar{x}+\OU u+\OV v(u)\),
are smooth and are tangent to $\U$ at $\bar{x}$, because  $v(u)=O(\|u\|^2)$. 
When, in addition,
the Hessian of $L_{\U}(u;0)$ exists at $u=0$ (see \cite[Definition 3.8]{ulagconv} and the preamble), 
the second-order expansion of $L_{\U}$ is possible \cite[Section 2.2]{vualg}. 
Lemma 3 of \cite{vualg} shows that in that case, the derivative of the trajectory provides
a $\mathcal{C}^1$ $\U$-gradient.  

The connection with the proximal point is given by the following very useful equivalence.

\begin{thm}\emph{\cite[Theorem 4]{vualg}}\label{thm:pri} Let $\chi(u)$ be a primal track leading to a minimizer $\bar{x}\in\R^n.$ Suppose that $0\in\partial f(\bar{x})$ and that we have a function $r(x)>0$ such that $r(x)\|x-\bar{x}\| \to 0$ when $x\to \bar{x}$.  Define $u_r(x)=(\Prox_f^r(x)-\bar{x})_{\U}.$ Then for all $x$ close enough to $\bar{x}$ and $r=r(x)$, we have that 
	$$\Prox_f^r(x)=\chi(u_r(x))=\bar{x}+u_r(x)\oplus v(u_r(x)).$$
Moreover, $u_r(x)\to0$ as $x\to\bar{x}.$
\end{thm}
\noindent In Theorem \ref{thm:pri}, $r(x)$ plays the role of a prox-parameter that can be dynamically selected within an algorithm (provided $r(x)\|x-\bar{x}\| \to 0$ when $x\to \bar{x}$).  The conclusion of Theorem \ref{thm:pri} allows us to concentrate on finding the proximal point instead of being concerned about how to find the primal track, since close to $\bar{x}$ they are one and the same.  Moreover, note that Theorem \ref{thm:pri} does not require $r(x)$ to be constant.  This provides valuable flexibility that greatly improves numerical performance in $\VU$-algorithms.  Finally, note that a routine for finding the proximal point of a convex function at a given point already exists \cite{kiwiel1990proximity}. We give a brief summary of the method next. 

\subsection{Computing proximal points}

Given a convex function $f$ and an initial point $y_0,$ at iteration $j$ of the bundle routine we choose any subgradient $g_j\in\partial f(y_j)$ and define the \emph{linearization error}:
$$E_j=E(x,y_j)=f(x)-f(y_j)-g_j^\top(x-y_j).$$Since $f$ is convex, $E_j\geq0$ for all $j.$ Using the fact that $f(z)\geq f(y_j)+g_j^\top(z-y_j)$ for all $z\in\R^n,$ we have that$$f(z)\geq f(x)+g_j^\top(z-x)-E_j\mbox{ for all }z\in\R^n.$$In other words, $g_j\in\partial_{E_j}f(x).$ The bundle $\{(E_j,g_j)\}_{j\in\C},$ where $\C$ is a set that indexes information from previous iterations, is used to construct a convex piecewise-linear function $\varphi_j$ that approximates and minorizes $f.$ Then the new iteration point $y_{j+1}=\Prox_{\varphi_j}^r(y_j)$ is found, and the process repeats. This method is proved in \cite{correaconvergence1993} to converge to $\Prox_f^r(y_0).$

The cutting-plane model $\varphi_j$ uses subgradient information that is not available in our case.
In the DFO setting, subgradients will be estimated by means of certain \emph{simplex gradients}
using functional information only; see Section~\ref{sec:approx}.
\subsection{The $\VU$-algorithm}\label{ssec:conceptualVU}

When a primal track exists, the $\VU$-algorithm takes a step approximately following the primal track by way of a \emph{predictor step ($\U$-step)}, which is a Newton-like step parallel to the $\U$-space, followed by a \emph{corrector step ($\V$-step)}, which is a bundle subroutine estimate of the proximal point in the $\V$-space. The $\V$-step outputs a potential primal track point, which is then checked and either accepted or rejected, depending on whether sufficient descent is achieved. We now state an abbreviated version of the conceptual $\VU$-algorithm presented in \cite{vualg}. \medskip

\noindent{\bf Conceptual $\VU$-algorithm}\medskip

\noindent Step 0: Initialize the starting point $x_0$, proximal parameter $r>0$, iteration counter $k=0$ and other parameters.\medskip

\noindent Step 1: Given $g\in\partial f(x_k),$ compute the $\VU$-decomposition with subspace bases $\OV$ and $\OU.$\medskip

\noindent Step 2: Compute an approximate proximal point $x_{k+1} \approx \Prox_f^r(x^k)$. Increment $k\mapsto k+1.$\medskip

\noindent Step 3: If $x_k$ does not show sufficient descent, then declare a null step and repeat Step 2 to higher precision. If $x_k$ does show sufficient descent, then check stopping conditions and either stop or continue to Step 4.\medskip

\noindent Step 4: Compute the $\U$-gradient $\nabla L_{\U}(0;g_{\V})$ and $\U$-Hessian $\nabla^2L_{\U}(0;g_{\V}).$  Take a $\U$-Newton step by solving
	$$\nabla^2L_{\U}(0;g_{\V})\Delta u=-\nabla L_{\U}(0;g_{\V})$$
for $\Delta u$ and setting
  $$x_{k+1}=x_k+\OU\Delta u.$$
Increment $k\mapsto k+1$, update $r$, and go to Step 1.\medskip

\noindent\textbf{End algorithm.}

\section{Defining inexact subgradients and related approximations}\label{sec:approx}

We now consider how to make implementable the conceptual $\VU$-algorithm in a derivative-free setting as provided by Assumptions \ref{ass:finitemax} through \ref{ass:affineindep}.  In order to prove convergence, we make use of the results of \cite{numvu} and \cite{hareplaniden2016}. We use the techniques in \cite{numvu} to approximate a subgradient, the $\VU$-decomposition, the $\U$-gradient and the $\U$-Hessian for the function $f$ at a point $\bar{x}$.\par
To define an inexact subgradient for $f,$ we make use of the \emph{simplex gradients} of each $f_i.$ The simplex gradient is defined as the gradient of the approximation resulting from a linear interpolation of $f$ over a set of $n+1$ points in $\R^n$ \cite{Kelley99}.
\begin{df}Let $Y = [ y_0, y_1, \dots ,y_n ]$ be a set of affinely independent points in $\R^n$.  Then it is said that \emph{$Y$ forms a simplex}, with {\em simplex diameter} $$\varepsilon= \max_{j=1, \dots n}\|y_j - y_0\|.$$
The {\em simplex gradient} of a function $f_i$ over $Y$ is given by
    $$\nabla_\varepsilon f_i(Y) = M^{-1} \delta f_i(Y),$$
where
    $$M= \bigg[ y_1 - y_0 \cdots y_n-y_0 \bigg]^\top, \mbox{ and}\quad\delta f_i(Y) = \left[ {\begin{array}{cc}
			f_i(y_1) - f_i(y_0)\\
			\vdots\\
			f_i(y_n)-f_i(y_0)
			\end{array}} \right].
    $$
The {\em condition number} of $Y$ is given by $\|\hat{M}^{-1}\|$, where
    $$\hat{M} = \frac{1}{\varepsilon} [y_1-y_0~~y_2-y_0~~ \dots ~~y_n-y_0]^\top.$$\end{df}
\noindent An important aspect of the condition number is that it is always possible to keep it bounded away from zero while simultaneously making $\varepsilon$ arbitrarily close to zero (see Remark \ref{rem:zero}). The following result provides an error bound for the distance between the simplex gradient and the exact gradient for a smooth function.

\begin{thm} \emph{\cite[Lemma 6.2.1]{Kelley99}}\label{errorbound}
Consider $f_i \in \mathcal{C}^{2}$.  Let $Y=[y_0,y_1, \dots, y_n]$ form a simplex.  Then there exists $\mu$ constant depending on $n$ and the local Lipschitz constant of $\nabla f_i$ such that
    \[\| \nabla_\varepsilon f_i(Y) - \nabla f_i(y_0)\| \le \varepsilon\mu \| \hat{M}^{-1} \|. \]
\end{thm}
\noindent We set $y_0=\bar{x},$ and $y_1$ through $y_n$ to $\bar{x}+\varepsilon e_i$, where $e_i$ is the $i$\textsuperscript{th} canonical vector. If desired, a rotation matrix can be used to prevent the $y_i$ vectors from being oriented in the coordinate directions every time. Now we define Subroutine \ref{sub1}, which we use to find an approximate subgradient $g^\varepsilon,$ approximations of the subspace bases $\OV$ and $\OU$ and the approximate $\U$-gradient $\nabla_\varepsilon L_{\U}(0;g^\varepsilon_{\V}).$
\begin{sub}[First-order approximations]\label{sub1}\color{white}.\color{black}\medskip

\noindent Step 0: Input $\bar{x}$ and $\varepsilon.$\medskip

\noindent Step 1: Set $Y=[\bar{x}~~~\bar{x}+\varepsilon e_1~~~\bar{x}+\varepsilon e_2~~~\cdots~~~\bar{x}+\varepsilon e_n].$\medskip

\noindent Step 2: Find $A(\bar{x})$ and calculate $\nabla_\varepsilon f_i(Y)$ for each $i\in A(\bar{x}).$\medskip

\noindent Step 3: Set\begin{itemize}[wide=0pt]
\item[(i)]$g^\varepsilon=\frac{1}{|A(\bar{x})|}\sum\limits_{i\in A(\bar{x})}\nabla_\varepsilon f_i(Y);$
\item[(ii)]$\OV$ to be the matrix of column vectors $$\nabla_\varepsilon f_i(Y)-\nabla_\varepsilon f_I(Y)$$ for each $i\in A(\bar{x})\setminus\{I\},$ where $I$ is the first element of $A(\bar{x});$\medskip
\item[(iii)]$\OU=\nul\OV/|\nul\OV|;$\medskip
\item[(iv)]$\nabla_\varepsilon L_{\U}(0;g^\varepsilon_{\V})=\OU^\top g^\varepsilon.$
\end{itemize}
\end{sub}\noindent\textbf{End subroutine.}

\begin{rem}\label{rem:zero}
Using $Y$ from Step 1, we have
$$\hat{M}=\frac{1}{\varepsilon}[\bar{x}+\varepsilon e_1-\bar{x}~~\bar{x}+\varepsilon e_2-\bar{x}~~\cdots \bar{x}+\varepsilon e_n-\bar{x}]=\Id,$$ so that $\|\hat{M}^{-1}\|=1$ while $\varepsilon$ can be arbitrarily small.
\end{rem}

\begin{rem}\label{rem:epsilongradnotation}  By fixing $Y$ in Step 1, we have $g^\epsilon$ and $\nabla_\varepsilon L_{\U}(0;g^\varepsilon_{\V})$ defined directly using $\epsilon$.  This is done primarily to simplify notation.  If a more flexible implementation is desired, the notation $g^\epsilon(Y)$ and $\nabla_\varepsilon L_{\U}(0;g^\varepsilon_{\V})(Y)$ could be employed.
\end{rem}

\noindent The following theorem shows that the outputs $g^\varepsilon$ and $\nabla_\varepsilon L_{\U}(0;g^\varepsilon_{\V})$ from Subroutine \ref{sub1} are good approximations.
\begin{thm}\label{thm:approximateVU}
Let $f:\R^n\to\R$ satisfy Assumptions \ref{ass:finitemax} and \ref{ass:affineindep}. Fix $\bar{x}\in\dom f.$ Then there exist $\mu$ constant depending on $\bar{x}$ and $g\in\ri\partial f(\bar{x})$ such that for $\varepsilon > 0$ sufficiently small, one can obtain
\begin{itemize}[wide=0pt]
\item[(i)] an approximate subgradient $g^\varepsilon$ such that $$\|g^\varepsilon - g\| \leq\varepsilon\mu, \|g^\varepsilon_{\U} - g_{\U}\| \leq\varepsilon\mu\mbox{ and }\|g^\varepsilon_{\V} - g_{\V}\| \leq\varepsilon\mu;$$
\item[(ii)] the approximate $\U$-gradient $\nabla_\varepsilon L_{\U} (0;g^\varepsilon_{\V})$ such that $$\|\nabla_\varepsilon L_{\U} (0;g^\varepsilon_{\V}) - \nabla L_{\U} (0;g_{\V})\|\leq\varepsilon\mu.$$
\end{itemize}
\end{thm}
\begin{proof} By Theorem \ref{errorbound} with $\|\hat{M}^{-1}\|=1$ as per Remark \ref{rem:zero}, there exists $\mu_1>0$ such that $\|\nabla_\varepsilon f_i(Y)-\nabla f_i(\bar{x})\|\leq\varepsilon\mu_1.$ Then by \cite[Lemma 2.6]{numvu}, there exist unique $\lambda_i\geq0$ with $\sum_{i\in A(\bar{x})}\lambda_i=1$ such that 
	$$g=\sum\limits_{i\in A(\bar{x})}\lambda_i\nabla f_i(\bar{x})\in\ri\partial f(\bar{x})
	\qquad\mbox{ and }\qquad 
	g^\varepsilon=\sum\limits_{i\in A(\bar{x})}\lambda_i\nabla_\varepsilon f_i(Y)\in\ri\partial_\varepsilon f(\bar{x})$$
are such that $\|g^\varepsilon-g\|\leq\varepsilon\mu_1.$ By \cite[Lem 4.3]{numvu} and \cite[Thm 5.3]{numvu}, we have the existence of $\mu_2,\mu_3>0$ such that $\|g^\varepsilon_{\U}-g_{\U}\|\leq\varepsilon\mu_2,$ $\|g^\varepsilon_{\V}-g_{\V}\|\leq\varepsilon\mu_2$ and $\|\nabla_\varepsilon L_{\U}(0;g^\varepsilon_{\V})-\nabla L_{\U}(0;g_{\V})\|\leq\varepsilon\mu_3.$ Setting $\mu=\max\{\mu_1,\mu_2,\mu_3\}$ completes the proof.
\end{proof}
\noindent Next, we find the approximate $\U$-Hessian $\nabla^2_\varepsilon L_{\U}(0;g^\varepsilon_{\V}),$ as outlined in \cite{numvu}. To do so, we need the Frobenius norm.
\begin{df}
The Frobenius norm $\|M\|_F$ of a matrix $M\in\R^{p\times q}$ with elements $a_{ij}$ is defined by
$$\|M\|_F=\sqrt{\sum\limits_{i=1}^p\sum\limits_{j=1}^qa_{ij}^2}.$$
\end{df}
\noindent We define the matrix $Z\in\R^{n\times(2n+1)}:$
	$$Z=[\bar{x}~~~\bar{x}+\varepsilon e_1~~~\bar{x}-\varepsilon e_1~~~\cdots~~~\bar{x}+\varepsilon e_n~~~\bar{x}-\varepsilon e_n].$$
To build an approximate Hessian of $f_i(\bar{x})$ for each $i\in A(\bar{x}),$ we solve the minimum Frobenius norm problem:
$$\{(H_i^*, D_i^*, C_i^*\}\in\argmin_{H, D, C}\left\{\|H_i\|_F:\frac{1}{2}z_j^\top H_i z_j + D_i^\top z_j+C_i=f_i(z_j) ~\mbox{for all} ~z_j \in Z\right\}$$ and set $\nabla^2_\varepsilon f_i(Z)=H_i^*$, where $z_j \in Z$ means $z_j$ is the $j$\textsuperscript{th} column of $Z$. The solution is obtained by solving a quadratic program. We then set
	$$H=\frac{1}{|A(\bar{x})|}\sum\limits_{i\in A(\bar{x})}\nabla^2_\varepsilon f_i(Z),$$
and define the approximate $\U$-Hessian of $f(\bar{x}):$ 
	$$\nabla^2_\varepsilon L_{\U}(0;g^\varepsilon_{\V})=\OU^\top H\OU.$$
The following result provides the error bound for the approximate Hessian.
\begin{thm}\emph{\cite[Theorem 6.1]{numvu}}\label{thm:hessbound}
Let $\bar{x}$ be fixed. Suppose that Assumption \ref{ass:affineindep} holds and that for any $\varepsilon>0$ there exists $\mu$ constant depending on $\bar{x}$ such that $\|\nabla_\varepsilon f_i(\bar{x})-\nabla f(\bar{x})\|<\varepsilon\mu$ and $\|\nabla^2_\varepsilon f_i(\bar{x})-\nabla^2f(\bar{x})\|<\varepsilon\mu.$ Then
	$$\|\nabla^2L_{\U}(0;g_{\V})-\nabla^2_\varepsilon L_{\U}(0;g_{\V}^\varepsilon)\|\leq\varepsilon\left[2\sqrt{2}\sqrt{|A(\bar{x})|-1}\| \OV^\dagger \|\|H\|(2\mu+\mu^2\varepsilon)+\mu\right],$$ 
where $\OV^\dagger$ represents the Moore-Penrose pseudo-inverse of $\OV$. 
Thus,
	$$\lim\limits_{\varepsilon\searrow0}\nabla^2_\varepsilon L_{\U}(0;g_{\V}^\varepsilon)=\nabla^2L_{\U}(0;g_{\V}).$$
\end{thm}
\noindent Now we state Subroutine \ref{sub2}, which is used to find the approximate $\U$-Hessian of $f$ at $\bar{x}.$
\begin{sub}[Second-order approximation]\label{sub2}\color{white}hi\color{black}\medskip

\noindent Step 0: Input $\bar{x},$ $\varepsilon,$ $A(\bar{x})$ and $\OU.$\medskip

\noindent Step 1: Set $Z=[\bar{x}~~~\bar{x}+\varepsilon e_1~~~\bar{x}-\varepsilon e_1~~~\cdots~~~\bar{x}+\varepsilon e_n~~~\bar{x}-\varepsilon e_n].$\medskip

\noindent Step 2: Calculate $\nabla^2_\varepsilon f_i(Z)$ for each $i\in A(\bar{x}).$\medskip

\noindent Step 3: Set $\nabla^2_\varepsilon L_{\U}(0;g^\varepsilon_{\V})=\OU^\top\left(\frac{1}{|A(\bar{x})|}\sum\limits_{i\in A(\bar{x})}\nabla^2_\varepsilon f_i(Z)\right)\OU.$\medskip

\end{sub}\noindent\textbf{End subroutine.}

\begin{rem} Similar to Subroutine \ref{sub1}, by fixing $Z$ in Step 1, there is no need to put it in the notation for our approximate $\U$-Hessian.  If a more flexible algorithm were desired, then the notation $\nabla^2_\varepsilon L_{\U}(0;g^\varepsilon_{\V})(Z)$ could be used.
\end{rem}

Theorems \ref{thm:approximateVU} and \ref{thm:hessbound} provide us with the tools needed to perform the approximate $\U$-step in the derivative-free $\VU$-algorithm.  In order to perform the approximate $\V$-step, we need to be able to approximate a proximal point in a derivative-free setting. A subroutine that accomplishes this, called the \emph{tilt-correct DFO proximal bundle method}, was introduced in \cite{hareplaniden2016}. Details are reproduced in Step 2 of the DFO $\VU$-algorithm below. At iteration $j$ of said subroutine, a subgradient is approximated by modelling $f$ with a piecewise-linear function $\varphi_j$ and then finding the proximal point of $\varphi_j$. This method is proved in \cite{hareplaniden2016} to converge to the desired proximal point within a preset tolerance. Theorem \ref{thm:approximateVU}(i) provides the approximate subgradients required for this step.\par The tilt-correct DFO proximal bundle method involves a possible correction to the approximate subgradient found at each iteration (Step 1.1 of the upcoming DFO $\VU$-algorithm), which ensures that the model function value at the current iterate $x_k$ is not greater than the objective function value at $x_k$. This is not a concern when exact subgradients are available, because then the model function naturally bounds the (convex) objective function from below, but when using approximate subgradients it is possible for the model function to lie partially above the objective function. In that case, tilting the linear piece down until the model and true function values are consistent at $x_k$ 
makes the model no worse \cite[Lemma 3.1]{hareplaniden2016}. The tilt procedure is explained in \cite[\S 3.1]{hareplaniden2016}. Suffice it to say here that once $g^\varepsilon$ is found, it can be replaced by the approximate subgradient defined by \eqref{tilteq}, which complies with all of our requirements.

\subsection{The DFO $\VU$-Algorithm}\label{DFOVU-alg}
In the following algorithm, we use $k$ for the outer counter and $j$ for the inner ($\V$-step subroutine) counter.  Henceforth, we refer to this algorithm as \Vd. \\

\noindent Step 0: \emph{Initialization.}  Choose a stopping tolerance $\delta \geq 0$, an accuracy tolerance $\varepsilon_{\tt min} \geq 0$ for the subgradient errors, a descent-check parameter $m\in(0,1)$ and a proximal parameter $r>0.$ Choose an initial point $x_0\in\dom f$ and an initial subgradient accuracy $\varepsilon_0\geq 0.$ Set $k=0.$\medskip

\noindent Step 1: \emph{$\V$-step.}
\begin{itemize}
\item[]Step 1.0: \emph{Initialization.} Set $j=0,$ $z_0=x_k$ and $\C_0=\{0\}.$
\item[]Step 1.1: \emph{Linearization.} Call Subroutine \ref{sub1} with input $(z_j,\varepsilon_k)$ to find $\tilde{g}_j^{\varepsilon_k}.$ Compute $E_j=f(z_j)+\tilde{g}_j^{\varepsilon_k\top}(z_0-z_j)-f(z_0)$ and set
\begin{equation}\label{tilteq}g_j^{\varepsilon_k}=\tilde{g}_j^{\varepsilon_k}+\max(0,E_j)\frac{z_0-z_j}{\|z_0-z_j\|^2}.\end{equation}
\item[]Step 1.2: \emph{Model.} Define
$$\varphi_j^{\varepsilon_k}(z)=\max\limits_{i\in\C_j}\left\{f(z_i)+g_i^{\varepsilon_k\top}(z-z_i)\right\}.$$
\item[]Step 1.3: \emph{Proximal Point.} Calculate $z_{j+1}=\Prox_{\varphi_j^{\varepsilon_k}}^r(z_0).$
\item[]Step 1.4: \emph{Stopping Test.} If $f(z_{j+1})-\varphi_j^{\varepsilon_k}(z_{j+1})\leq\varepsilon_k^2/r,$ set $x_{k+1}=z_{j+1},$ calculate the aggregate subgradient of the model function: $s_{k+1}=r(z_0-z_{j+1})$, and go to Step 2.\medskip
\item[]Step 1.5: \emph{Update and Loop.} Create the aggregate bundle element $$(z_{j+1},\varphi_j^{\varepsilon_k},r(z_0-z_{j+1})).$$ Create $\C_{j+1}$ such that $\{-1,0,j+1\}\subseteq \C_{j+1}\subseteq\{-1,0,1,2,\cdots,j+1\}.$ Increment $j\mapsto j+1$ and go to Step 1.1.
\end{itemize}
\noindent Step 2: \emph{Stopping Test.} If $\|s_{k+1}\|^2\leq\delta\mbox{ and }\varepsilon_k\leq\varepsilon_{\tt min},$ output $x_{k+1}$ and stop.\medskip

\noindent Step 3: \emph{Update and Loop.}
\begin{itemize}
\item[]Case 3.1: If $f(x_k)-f(x_{k+1})\geq\frac{m}{2r}\|s_{k+1}\|^2$ and $\|s_{k+1}\|^2\leq\delta$ and $\varepsilon_k>\varepsilon_{\tt min},$ declare SERIOUS STEP and set $\varepsilon_{k+1}=\varepsilon_k/2.$\medskip
\item[]Case 3.2: If $f(x_k)-f(x_{k+1})\geq\frac{m}{2r}\|s_{k+1}\|^2$ and $\|s_{k+1}\|^2>\delta,$ declare SERIOUS STEP and set $\varepsilon_{k+1}=\varepsilon_k.$\medskip
\item[]Case 3.3: If $f(x_k)-f(x_{k+1})<\frac{m}{2r}\|s_{k+1}\|^2,$ declare NULL STEP and set $\varepsilon_{k+1}=\varepsilon_k/2.$
\end{itemize}
Increment $k\mapsto k+1.$ If SERIOUS STEP, go to Step 4. If NULL STEP, go to Step 1.\medskip

\noindent Step 4: \emph{$\U$-step.} Call Subroutine \ref{sub1} with input $(x_k,\varepsilon_k)$ to find $A(x_k),$ $g_k^{\varepsilon_k},$ $\OU_k$ and $\nabla_\varepsilon L_{\U}(0;(g_k^{\varepsilon_k})_{\V}).$ Call Subroutine \ref{sub2} with input $(x_k,\varepsilon_k,A(x_k),\OU_k)$ to find $\nabla^2_\varepsilon L_{\U}(0;(g_k^{\varepsilon_k})_{\V}).$ Compute an approximate $\U$-quasi-Newton step by solving the linear system
\begin{equation}\label{eq:ustepalg}\nabla^2_\varepsilon L_{\U}(0;(g_k^{\varepsilon_k})_{\V})\Delta u_k=-\nabla_\varepsilon L_{\U}(0;(g_k^{\varepsilon_k})_{\V})\end{equation}
for $\Delta u_k.$ Set $x_{k+1}=x_k+\OU_k\Delta u_k$ and $\varepsilon_{k+1}=\varepsilon_k.$ Increment $k\mapsto k+1$ and go to Step 1.\medskip

\noindent\textbf{End algorithm.}

\begin{note}
In Step 0, the stopping tolerance $\delta$ and accuracy tolerance $\varepsilon_{\min}$ can be set to $0$.  Setting these values to $0$ effectively makes the algorithm run without stopping conditions.  This allows for theoretical analysis of the algorithm, but, of course, these values should never be used in practice.  
\end{note}

\begin{note}
In Step 1.1, the call to Subroutine \ref{sub1} yields the active set, approximate $\U$-basis and approximate $\U$-gradient in addition to $\tilde{g}_j^{\varepsilon_k}.$ However, $\tilde{g}_j^{\varepsilon_k}$ is the only information we use from Subroutine \ref{sub1} in the $\V$-step, so we do not mention the other outputs in the statement of the algorithm.
\end{note}

\begin{note}
The $\varepsilon_k$ and the iteration counter $k$ are updated in Step 3 and again in Step 4 (if applicable). 
 This is explained by the fact that Step 4 is not called at every iteration.  An alternate formatting of the algorithm might have at the start of each iteration a decision on whether to do a $\V$-step or a $\U$-step.  Iterations that are $\V$-steps are frequent and can occur multiple times in a row.  This is captured in Step 3.  Iterations that are $\U$-steps only occur after successful $\V$-steps, and only in batches of one (i.e., a $\U$-step is never followed by another $\U$-step).  This is captured in Step 4.  
\end{note}

\subsection{Theoretical comparison to other DFO methods}\label{ssec:comparison}

Relative to other DFO methods, the \textbf{DFO $\VU$-algorithm} falls under the category of a model-based method \cite[Part 4]{AudetHare_DFObook}.  In this case, it uses function calls to construct a model of the objective function and then applies a $\VU$-style method to the model function. 

Most other DFO methods for nonsmooth optimization fall under the catergory of direct search methods \cite[Part 3]{AudetHare_DFObook}.  Direct search methods work by setting an incumbent solution and then polling around the incumbent solution to seek a point that provides a better function value.  If improvement is found, then the incumbent solution is updated, otherwise the algorithm reduces the polling radius and repeats.  If polling is done carefully, then convergence to a  critical point can be proven, even for nonsmooth functions \cite[Chpt 7]{AudetHare_DFObook}. 
 These ideas are the core of the Mesh Adaptive Direst Search (MADS) algorithm developed in \cite{AbramsonAudetDennis09,AudetBechardLeDigabel08,AudetDennis06} (among other papers).  We mention the MADS algorithm, as we use it as one basis of comparison in Section \ref{sec:numerics}.

A few derivative-free model-based methods for nonsmooth optimization have arisen in the last decade.   The first such approach appeared in 2008, in the work of Bagirov, Karas\"osen and Sezer \cite{BagirovKarasSezer08}.  The method proceeds by constructing a large number of approximate gradients at points near the incumbent solution, and using these approximate gradients to build a approximation of the subdifferential.  The approximate subdifferential is then used to drive a conjugate subgradient style algorithm.  This methods was implemented and tested under the name DGM.  The authors show the method can achieve 4 digits of accuracy, but do not include information on the number of function evaluations used, nor provide software.  As such, direct comparison to this method is not possible. 

A similar idea was proposed by Kiwiel \cite{Kiwiel10}.  In Kiwiel's approach, a large number of approximate gradients is used to construct an approximate subdifferential, and the approximate subdifferential is used in a gradient sampling style algorithm.  Only a theoretical development of this algorithm was presented.

In relation to \cite{BagirovKarasSezer08} and  \cite{Kiwiel10}, the algorithm herein also generates a collection of gradient approximations and uses them to construct nonsmooth first-order objects.  However, the algorithm herein uses the grey-box structure of the problem to control the construction of these approximation gradients.  In particular, the number of approximate gradients constructed at an incumbent solution is guided by the number of active indices at that point.  Furthermore, the algorithm herein uses the approximate gradients to approximate both subdifferentials and $\VU$-structure in the problem.  This sets our algorithm distinctly apart from these previous works.

Between direct-search methods and model-based methods lies the implicit filtering approach of Kelley \cite{Kelley11}.  The implicit filtering approach can be thought of as beginning with a direct-search poll step, but if success occurs, then instead of simply accepting the new point, the poll information is used to construct approximate gradients and a line search is applied to seek better improvement.  Convergence of the implicit filtering algorithm is based on a (locally) smooth objective function.

\section{Convergence}\label{sec:conv}

In this section, we examine the convergence of the DFO $\VU$-algorithm, starting with the $\V$-step. By \cite[Corollary 4.6]{hareplaniden2016}, if the $\V$-step never terminates, then
$$\lim\limits_{j\rightarrow\infty}\|z_{j+1}-z_j\|=0.$$
 Then \cite[Theorem 4.9]{hareplaniden2016} states that if $f$ is locally $K$-Lipschitz (which a finite-max function is), then
\begin{equation}\label{conveq3}
\|z_{j+1}-z_j\|\leq\frac{\varepsilon_k^2}{r(K+2\varepsilon_k)}\Rightarrow f(z_{j+1})-\varphi_j^{\varepsilon_k}(z_{j+1})\leq\frac{\varepsilon_k^2}{r}
\end{equation}
and the routine terminates. The properties of $\varphi_j^{\varepsilon_k}$ established in \cite[Lemma 4.1]{hareplaniden2016} show that if the $\V$-step with input $z_0$ stops at iteration $j$ and outputs $z_{j+1},$ then
$$\dist(\Prox_f^r(z_0),z_{j+1})\leq\frac{(\mu+1)\varepsilon_k}{r},$$
where $\mu$ is the constant of Theorem \ref{thm:approximateVU}. Now in order to prove the convergence of the main algorithm, we show that either the algorithm terminates in a finite number of steps or, in the case where no stopping occurs, 
 $\varepsilon_k\to0$ and $\liminf\|s_k\|\to 0$.  In either case, we arrive at a good approximation of the minimizer of $f.$ To accomplish that goal, we need the following definitions.
\begin{df}\label{derdef}
Let $\varepsilon\geq0.$ The \emph{$\varepsilon$-directional derivative} of $f$ at $x$ in direction $d$ is defined
$$f'_\varepsilon(x;d)=\inf\limits_{t>0}\frac{f(x+td)-f(x)+\varepsilon}{t}=\max\limits_{s\in\partial_\varepsilon f(x)}\{s^\top d\}.$$
\end{df}
\begin{df}\label{e1e2df}
Let $\varepsilon,\eta\geq0.$ A vector $v$ is an \emph{$(\varepsilon,\eta)$-subgradient} of $f$ at $\bar{x},$ denoted $v\in\partial_\varepsilon^\eta f(\bar{x}),$ if for all $x,$
$$f(x)\geq f(\bar{x})+v^\top(x-\bar{x})-\eta\|x-\bar{x}\|-\varepsilon.$$
\end{df}
\noindent Notice that by setting $\eta=0$ we recover the definition of the $\varepsilon$-subgradient, and by setting $\varepsilon=\eta=0$ we have the convex analysis subgradient. The next lemma provides enlightenment on the $(\varepsilon,\eta)$-subgradient in the general case.
\begin{lem}\label{gradequiv}
Let $\varepsilon,\eta\geq0$ and $f$ be convex with $\bar{x}\in\dom f$. Then
\begin{equation}\label{gradequiv1}
g\in\partial_\varepsilon^\eta f(\bar{x})\Leftrightarrow g\in\partial_\varepsilon f(\bar{x})+B_\eta.
\end{equation}
\end{lem}
\begin{proof}
$(\Rightarrow)$ Suppose $g\in\partial_\varepsilon^\eta f(\bar{x}).$ Since $\partial_\varepsilon f$ is closed and convex \cite[Theorem 1.1.4]{hiriart1993convex2}, we define
$$\bar{g}=\Proj_{\partial_\varepsilon f(\bar{x})}(g)$$
and we have $\bar{g}\in\partial_\varepsilon f(\bar{x}).$ Set $v=g-\bar{g},$ so that $g=\bar{g}+v,$ and for $t>0$ we use $x=\bar{x}+tv$ in the definition of the $(\varepsilon,\eta)$-subgradient:
\begin{align}
f(\bar{x}+tv)&\geq f(\bar{x})+(\bar{g}+v)^\top tv-\eta\|tv\|-\varepsilon,\nonumber\\
\frac{f(\bar{x}+tv)-f(\bar{x})+\varepsilon}{t}-v^\top\bar{g}&\geq\|v\|^2-\eta\|v\|,\nonumber\\
\inf\limits_{t>0}\frac{f(\bar{x}+tv)-f(\bar{x})+\varepsilon}{t}-v^\top\bar{g}&\geq\|v\|^2-\eta\|v\|,\nonumber\\
f'_{\varepsilon_1}(\bar{x};v)-v^\top\bar{g}&\geq\|v\|^2-\eta\|v\|.\label{finalone}
\end{align}
By the Projection Theorem, we have
$$p=\Proj_{\partial_\varepsilon f(\bar{x})}y\Leftrightarrow(y-p)^\top(z-p)\leq0\mbox{ for all }z\in\partial_\varepsilon f(\bar{x}).$$
So for all $\tilde{g}\in\partial_\varepsilon f(\bar{x})$ we have
\begin{align*}
(g-\bar{g})^\top(\tilde{g}-\bar{g})&\leq0,\\
v^\top\tilde{g}&\leq v^\top\bar{g}.\\
\end{align*}
Hence,
$$v^\top\bar{g}=\sup\limits_{\tilde{g}\in\partial_\varepsilon f(\bar{x})}\{v^\top\tilde{g}\}.$$ Using this together with Definition \ref{derdef}, \eqref{finalone} becomes
\begin{align*}
\|v\|^2-\eta\|v\|&\leq0,\\
\|v\|&\leq\eta.
\end{align*}
Therefore, $v\in B_\eta$, and we have $g=\bar{g}+v\in\partial_\varepsilon f(\bar{x})+B_\eta.$\\
$(\Leftarrow)$ Suppose that $g\in\partial_\varepsilon f(\bar{x})+B_\eta.$ Set $g=\bar{g}+v$ where $\bar{g}\in\partial_\varepsilon f(\bar{x})$ and $v\in B_\eta.$ Then by the definition of $\varepsilon$-subgradient and the Cauchy-Schwarz inequality, we have
\begin{align*}
f(x)-f(\bar{x})-g^\top(x-\bar{x})&=f(x)-f(\bar{x})-\bar{g}^\top(x-\bar{x})-v^\top(x-\bar{x}),\\
&\geq-\varepsilon-v^\top(x-\bar{x}),\\
&\geq-\varepsilon-\|v\|\|x-\bar{x}\|,\\
&\geq-\varepsilon-\eta\|x-\bar{x}\|.
\end{align*}
Therefore, $g\in\partial_\varepsilon^\eta f(\bar{x}).$
\end{proof}
\noindent Now we are ready to show that the inexact aggregate subgradient at any step is a good approximation of a true subgradient.
\begin{lem}\label{skplus1lem} Let $f$ satisfy Assumptions \ref{ass:finitemax} and \ref{ass:affineindep}.  Let $K$ be the Lipschitz constant of $f$.  If  \Vd\/ at iteration $k$ gives output $(x_{k+1},s_{k+1}),$ then
$$s_{k+1}\in\partial_{\frac{\varepsilon_k^2}{r}}^{\varepsilon_k K}f(x_{k+1}).$$\end{lem}
\begin{proof}
In \cite[(4.3)]{hareplaniden2016}, it is shown that 
	$$f(x)+\varepsilon_k K\|\hat{M}^{-1}\|\|x-x_{k+1}\|\geq\varphi^{\varepsilon_k}_j(x_{k+1})+s_{k+1}^\top(x-x_{k+1}).$$
Remark \ref{rem:zero} shows that $\|\hat{M}^{-1}\|=1$.  Since iteration $k$ has completed, the stopping test in Step 1.4 has passed\footnote{Recall that \cite[Corollary 4.6]{hareplaniden2016} and \cite[Theorem 4.9]{hareplaniden2016} ensure this happens in finite time.}, thus 
	$$\varphi_j^{\varepsilon_k}(x_{k+1})-f(x_{k+1})\geq-\frac{\varepsilon_k^2}{r}.$$ 
This implies
\begin{align*}
f(x)&\geq\varphi^{\varepsilon_k}_j(x_{k+1})-f(x_{k+1})+f(x_{k+1})+s_{k+1}^\top(x-x_{k+1})-\varepsilon_k K\|x-x_{k+1}\|\\
&\geq-\frac{\varepsilon_k^2}{r}+f(x_{k+1})+s_{k+1}^\top(x-x_{k+1})-\varepsilon_k K\|x-x_{k+1}\|.
\end{align*}
Thus, $s_{k+1}\in\partial_{\frac{\varepsilon_k^2}{r}}^{\varepsilon_k K}f(x_{k+1})$ by Definition \ref{e1e2df}.
\end{proof}
\noindent There are two special cases of Lemma \ref{skplus1lem} that are of interest; we consider what happens when the aggregate subgradient is zero and when the maximum subgradient error is also zero.

\begin{cor}\label{sk0cor} Let $f$ satisfy Assumptions \ref{ass:finitemax} and \ref{ass:affineindep}.  Let $K$ be the Lipschitz constant of $f$. If at iteration $k$ \Vd\/ gives output $(x_{k+1},s_{k+1}),$ then the following hold.
\begin{itemize}[wide=0pt]
\item[(i)] If $s_{k+1}=0,$ then $0\in\partial_{\frac{\varepsilon_k^2}{r}}f(x_{k+1})+B_{K\varepsilon_k},$ and by Lemma \ref{skplus1lem} we have that for all $x\in\R^n,$$$f(x)\geq f(x_{k+1})-\varepsilon_k K\|x-x_{k+1}\|-\frac{\varepsilon_k^2}{r}.$$
\item[(ii)] If $\varepsilon_k=s_{k+1}=0,$ then $0\in\partial f(x_{k+1})$ and $x_{k+1}$ is a minimizer of $f.$
\end{itemize}
\end{cor}

\begin{note}Since $\varepsilon_k=0$ can only occur if $\varepsilon_0=0$, item (ii) could alternately be stated using ``If $\varepsilon_0=s_{k+1}=0$''.
\end{note}

\noindent Now we need to consider the possibility that the algorithm does not terminate and what the effect would be. We begin with the scenario where an infinite number of serious steps is taken.

\begin{thm}\label{finitesuccesseslem} Let $f$ satisfy Assumptions \ref{ass:finitemax}, \ref{ass:compactlowerlevel}, and \ref{ass:affineindep}.  Suppose \Vd\/ is run without stopping conditions (i.e., $\delta = \varepsilon_{\min} = 0$).  If there is an infinite number of serious steps taken in Step 3, then $\varepsilon_k\rightarrow0$ and $\liminf_{k\rightarrow\infty}\|s_k\|=0.$
\end{thm}

\begin{proof}
Note that $f$ is bounded below, due to Assumption \ref{ass:compactlowerlevel}. Suppose that out of the infinite number of serious steps, $\|s_{k+1}\|^2$ is bounded away from $0$. That is, suppose there exists $\hat{\delta}>0$ such that $\|s_{k+1}\|^2> \hat{\delta}$ whenever $f(x_k)-f(x_{k+1})\geq\frac{m}{2r}\|s_{k+1}\|^2$, and $f(x_k)-f(x_{k+1})\geq\frac{m}{2r}\|s_{k+1}\|^2$ occurs an infinite number of times.  Then we have
	$$f(x_k)-f(x_{k+1})\geq\frac{m}{2r}\|s_{k+1}\|^2>\frac{m\hat{\delta}}{2r}$$
an infinite number of times. Since $\frac{m\hat{\delta}}{2r}$ is constant, we have $$\lim_{k\rightarrow\infty}[f(x_0)-f(x_k)]=\infty,$$ which contradicts the fact that $f$ is bounded below. Hence, eventually $\|s_{k+1}\|^2\leq\hat{\delta}$, so $\liminf_{k\rightarrow\infty}\|s_k\|=0$.  

Since we are supposing that the algorithm does not stop, we must have $\varepsilon_k>\varepsilon_{\tt min}=0$ and we set $\varepsilon_{k+1}=\varepsilon_k/2.$ This happens an infinite number of times, which gives $\varepsilon_k\rightarrow0.$ 
\end{proof}
\noindent Next comes the scenario where a finite number of serious steps is taken, yet the algorithm does not terminate.

\begin{lem}\label{sklem} Let $f$ satisfy Assumptions \ref{ass:finitemax}, \ref{ass:compactlowerlevel}, and \ref{ass:affineindep}.   Suppose \Vd\/ is run without stopping conditions (i.e., $\delta = \varepsilon_{\min} = 0$).  If there is a finite number of serious steps taken in Step 3, then for all $k$ sufficiently large,
\begin{equation}\label{eqs1}
\varepsilon_k>\left(1-\frac{m}{2}\right)^{1/2}\|s_{k+1}\|.
\end{equation}
\end{lem}

\begin{proof} Let $\bar{k}$ be the final iteration where a serious step occurs, so that a null step occurs at every $k>\bar{k}.$ Since $s_{k+1}=r(x_k-x_{k+1})$ is the aggregate subgradient of the model function $\varphi_j^{\varepsilon_k}$ at $z_{j+1}=x_{k+1},$ we have $s_{k+1}\in\partial\varphi_j^{\varepsilon_k}(x_{k+1})$ \cite[Lemma 4.1(c)]{hareplaniden2016}. Thus,
$$\varphi_j^{\varepsilon_k}(x)\geq\varphi_j^{\varepsilon_k}(x_{k+1})+s^\top_{k+1}(x-x_{k+1})\mbox{ for all }x.$$
By the tilt-correction ($E_j$ in Step 1.1), we have that at $x_k,$
\begin{align}
f(x_k)&\geq\varphi_j^{\varepsilon_k}(x_{k+1})+s^\top_{k+1}(x_k-x_{k+1})\mbox{ \cite[Lemma 4.1(b)]{hareplaniden2016}},\nonumber\\
&=\varphi_j^{\varepsilon_k}(x_{k+1})+\frac{1}{r}s^\top_{k+1}[r(x_k-x_{k+1})],\nonumber\\
&=\varphi_j^{\varepsilon_k}(x_{k+1})+\frac{1}{r}\|s_{k+1}\|^2,\nonumber\\
&=\varphi_j^{\varepsilon_k}(x_{k+1})-f(x_{k+1})+f(x_{k+1})+\frac{1}{r}\|s_{k+1}\|^2,\nonumber\\
f(x_k)-f(x_{k+1})&\geq\varphi_j^{\varepsilon_k}(x_{k+1})-f(x_{k+1})+\frac{1}{r}\|s_{k+1}\|^2.\label{eqs2}
\end{align}
By the stopping test in Step 1.4, we have
\begin{equation}\label{eqs3}
\varphi_j^{\varepsilon_k}(x_{k+1})-f(x_{k+1})\geq-\frac{\varepsilon_k^2}{r}.
\end{equation}
Combining \eqref{eqs2} and \eqref{eqs3} yields
\begin{equation}\label{eqs4}
f(x_k)-f(x_{k+1})\geq\frac{1}{r}\|s_{k+1}\|^2-\frac{\varepsilon^2_k}{r}.
\end{equation}
Then for all $k>\bar{k},$ by Step 3(iii) and \eqref{eqs4} we have
\begin{align*}
\frac{m}{2r}\|s_{k+1}\|^2&>f(x_k)-f(x_{k+1}),\\
\frac{m}{2r}\|s_{k+1}\|^2&>\frac{1}{r}\|s_{k+1}\|^2-\frac{\varepsilon^2_k}{r},\\
\varepsilon^2_k&>\left(1-\frac{m}{2}\right)\|s_{k+1}\|^2,
\end{align*}
and \eqref{eqs1} is proved.
\end{proof}
\begin{cor}\label{infinitefailurescor}Let $f$ satisfy Assumptions \ref{ass:finitemax}, \ref{ass:compactlowerlevel}, and \ref{ass:affineindep}.   Suppose \Vd\/ is run without stopping conditions (i.e., $\delta = \varepsilon_{\min} = 0$).  If there is a finite number of serious steps taken in Step 3, then $\varepsilon_k\rightarrow0$ and $\|s_k\|\rightarrow0.$
\end{cor}
\begin{proof}
Since there is a finite number of successes and the algorithm does not terminate, there is an infinite number of failures. By Step 3(iii), $\varepsilon_k\rightarrow0.$ By \eqref{eqs1}, $\|s_k\|\rightarrow0.$
\end{proof}
\noindent Notice, if $\delta=\varepsilon_{\tt min}=0$, then by Step 2 of the algorithm we see that the only way it will terminate is if $s_{k+1}=0$ and $\varepsilon_k=0$.  Since $\varepsilon_k>0$, this cannot occur.  (If it could occur, then    Corollary \ref{sk0cor} would imply $x_{k+1}$ is a minimizer of $f$.) 

\noindent Theorem \ref{bigthm} below unites the convergence results of this section.

\begin{thm}\label{bigthm}Let $f$ satisfy Assumptions \ref{ass:finitemax}, \ref{ass:compactlowerlevel}, and \ref{ass:affineindep}.  Suppose \Vd\/ is run on $f$ and generates the sequence $\{x_k\}.$ Then either the algorithm terminates at some iteration $\bar{k}$ with $\|s_{\bar{k}+1}\|\leq\sqrt{\delta}$ and $\varepsilon_{\bar{k}}\leq\varepsilon_{\tt min},$ or $\liminf_{k\to\infty}\|s_k\|=0$ and $\varepsilon_k\to0.$ In the latter case, any cluster point $\bar{x}$ of a subsequence $\{x_{k_j}\}$ such that $s_{k_j}\to0$ satisfies $0\in\partial f(\bar{x}).$
\end{thm}

\begin{proof}
If the algorithm terminates at iteration $\bar{k},$ by Step 2 we have $\|s_{\bar{k}+1}\|\leq\sqrt{\delta}$ and $\varepsilon_{\bar{k}}\leq\varepsilon_{\tt min}.$ Suppose the algorithm does not terminate. If there is an infinite number of serious steps taken, then $\liminf_{k\to\infty}\|s_k\|=0$ and $\varepsilon_k\to0$ by Theorem \ref{finitesuccesseslem}. If there is a finite number of serious steps taken, then $s_k\to0$ and $\varepsilon_k\to0$ by Corollary \ref{infinitefailurescor}. In either case, consider any subsequence $\{x_{k_j}\}$ with cluster point $\bar{x}$ such that $s_{k_j}\to0.$ Since $\varepsilon_k\to0,$ we have $\varepsilon_{k_j}\to0.$ By \cite[Proposition 5]{asplund1969gradients}, we have that the $(\varepsilon_{k_j}^2/r)$-subdifferential of $f$ is continuous jointly as a function of $(x,\varepsilon_{k_j})$ on $(\ri\dom f)\times(0,\infty).$ Since $B_{K\varepsilon_{k_j}}$ is continuous as well, by \cite[Proposition 3.2.7-3]{nel2016continuity} and \eqref{gradequiv1} we have that $\partial_{\varepsilon_{k_j}^2/r}^{\varepsilon_{k_j} K}$ is continuous. Therefore, since $s_{k_j+1}\in\partial^{\varepsilon_{k_j}K}_{\varepsilon^2_{k_j}/r}f(x_{k_j+1})$ by Lemma \ref{skplus1lem}, $s_{k_j+1}\to0,$ $\varepsilon_{k_j}\to0$ and $x_{k_j+1}\to\bar{x},$ we have $0\in\partial f(\bar{x}).$
\end{proof}

\section{Numerical Results}\label{sec:numerics}

In this section, we present some numerical tests using \Vd. The tests were run on a 3.2 GHz Intel Core i5 processor with a 64-bit operating system, using MATLAB version 9.3.0.713579 (R2017b).
Our testing includes the following algorithms.
 \begin{enumerate}[wide=0pt]
\item \Vd\/ using the default of $2n+1$ function calls per Hessian approximation;
\item  \Vc, an inexact bundle method along the lines of \cite{deOliveiraSagaLemar2014}, with access to the grey-box available to {\sc DFO-VU}: the value function is exact and the subgradient is approximated by means of the DFO approach in Subroutine \ref{sub1};
\item \Vb, a classical bundle method in proximal form \cite[Part II]{bonnans-gilbert-lemarechal-sagastizabal-2006};
\item \Va, the Composite Bundle method from \cite{sagastizabal-2013};
\item \Ve, a well-established DFO method from \cite{le2009nomad}.  
\end{enumerate}
Algorithms \Vd\/, \Vc\/, \Vb\/, and \Va\/ are all bundle-style algorithms, while \Ve\/ is a DFO solver.  Algorithms \Vb\/ and \Va\/ use \emph{exact subgradient information}. As such, we expect those solvers to outperform both \Vd\/ and \Vc.  These inexact variants, by contrast, are on equal ground and we expect to see a positive impact of the $\mathcal{VU}$-decomposition in terms of accuracy.

We use two sets of test problems.  The first test set contains convex problems that satisfy the assumptions used in the proof of convergence.  The second test set contains nonconvex problems that do not satisfy the assumptions used in the proof of convergence.  This test set was done solely for prospective illustrative purposes, since the convergence of \Vd\/ for
nonconvex functions is beyond the scope of this work.  As such, for the nonconvex problems we only examine the behaviour of \Vd\/.

\subsection{Convex test functions and benchmark rules for the bundle solvers}\label{ss-benchrules}

We considered 301 max-of-quadratics functions. The first one is the classical {\sc maxquad} function in nonsmooth optimization \cite[Part II]{bonnans-gilbert-lemarechal-sagastizabal-2006}, for which the dimension is $n=10$, the optimal value is $\bar f=-0.84140833459641814$, and  $\dim\mathcal{V}$ at a solution is equal to 3. The remaining 300 problems were generated randomly in dimensions $n\in\{10,20,30,40,50\}$.  Each problem is generated such that the minimizer is $\bar x=0\in\R^n$ with $\bar f=0$.  The problems are designed with various final $\mathcal{V}$-dimensions $\dim\mathcal{V}\in\{0.25n, 0.5n, 0.75n\}.$ The functions were generated as follows; given $m\geq|A(\bar x)|=\dim\mathcal{V}+1$, 
\begin{equation}\label{def-testf}
f(x)=\max\limits_{j\in\{1,2,\ldots,m\}}\left\{\frac 12 x^\top H_j x+b_j^\top x\right\},\end{equation}
for random $H_j\in S^n_+$ and $b_j\in\R^n$. The symmetric, positive semidefinite matrices $H_j$ have condition number equal to $(\rank H)^2=(\dim\mathcal{V})^2$, and the set of vectors $\{b_2-b_1,\ldots,b_{\dim\mathcal{V}+1}-b_1\}$ is linearly independent. The above setting guarantees that all the assumptions in Section~\ref{ssec:assumptions} hold for the considered instances.\par We must acknowledge and accept that some of the inner workings of each solver make it difficult to compare the results fairly. First, \Va ~and \Vb ~make blackbox-calls ({\tt bb}-calls) that yield exact values for the function and a subgradient, while \Vc ~and \Vd ~call a grey-box that yields exact function values and approximate subgradients. Second, to avoid machine error due to a near-singular matrix in the second-order approximation created in Subroutine \ref{sub2}, \Vd\ ~stops  when in Step 4 the parameter $\varepsilon_k$ becomes smaller than $10^{-5}$. Third, \Vc ~stops when there are more than 18 consecutive noise-attenuation steps; we refer the reader to \cite{deOliveiraSagaLemar2014} for details. Barring the above, the parameters for \Va, \Vb, and \Vc ~are those chosen for the Composite Bundle solver in \cite{sagastizabal-2013}. In an effort to make the comparisons as fair as possible, we adopted the following rules.
\begin{enumerate}[wide=0pt]
\item All solvers use the same quadratic programming built-in MATLAB solver, {\tt quadprog}.\medskip
\item For all solvers, the stopping tolerance was set to $10^{-2}$, which for \Vd ~means that in Step 2, $\delta=\varepsilon_{\min}=10^{-2}$.\medskip
\item The maximum number of ${\tt bb}$-calls was set to {\maxs}$= 800\min(n,20)$. This corresponds to function and subgradient evaluations for the exact variants and to function evaluations for the inexact variants.\medskip
\item For all solvers, a run was declared a failure when \maxs ~was reached or when there was an error in the QP solver.\medskip
\item The methods use the same starting points, with components randomly drawn in $[-1,1]$. We ran all the instances with two starting points, for a total of $602$ runs.
\end{enumerate}
For those readers interested in implementing \Vd,
we mention the following additional numerical tweaks that had a positive impact
in the algorithm's performance.
\begin{enumerate}[wide=0pt]
\item In the $\mathcal{U}$-step, finding the active index set $A(x_k)$
in Subroutine~\ref{sub1}  is tricky. We note that using an absolute criterion ($f_i(x_k)=f(x_k)$)
was worse than the following soft-thresholding test:
\begin{equation}\label{eq:almostactive}
i\in A(x_k) \mbox{ when }
f(x_k)- f_i(x_k)\leq 0.001|f(x_k)|\,.
\end{equation}
\item In Step 1.3, it is often preferable to calculate the proximal point
$z_{j+1}$ by solving the dual of the quadratic programming problem defining
$\Prox_{\varphi_j^{\varepsilon_k}}^r(z_0)$.
\item   The tilting of gradients in \eqref{tilteq} is done only when
$E_j$ is larger than $10^{-8}.$ Otherwise, we set $g_j^{\varepsilon_k}=\tilde{g}_j^{\varepsilon_k}.$\medskip
\item  As long as the proximal parameter remains uniformly bounded, it can vary along iterations without impairing
the convergence results.
We have found the following rule to be effective, and use it in our testing,
\[ t_k=\left\{\begin{array}{ll}
\frac{0.5|g(x_k)|^2} {1+|f(x_k)|},
&\mbox{if }|f(x_k)|>10^{-10},\\
2,&\mbox{otherwise,}\end{array}\right.\]
and let
\[ r_k=\max\left\{1,\min\left\{\frac{1}{t_k},100r_{k-1},10^6\right\}\right\}\,.\]
\item In Step 2.5, the new bundle $\mathcal{B}_{j+1}$ keeps \emph{almost
active} indices. As can be seen from \eqref{eq:almostactive}, we accept as active the subfunctions that are close to active at each iteration point, so as not to dismiss those that are active but do not quite appear to be so because of numerical error.
\end{enumerate}

\subsection{Benchmark of bundle solvers}

We first describe the indicators defined to compare the solvers. The number of iterations is not a meaningful measure for comparison, because each solver involves a very different computational effort per iteration. This depends not only on the solver, but also on how many evaluations are done per iteration. Moreover, since the exact variants do not spend calls to make the DFO subgradient approximation, neither the total solving time nor the number of {\tt bb}-calls are meaningful measures. As the optimal values are known for the considered instances, we compare the accuracy reached by each solver. Denoting the best function value of the analyzed case by $f^{\tt found},$
\[\mbox{\lttRA}=
\max\left[0,- \log_{10} \max\left(10^{-16},\frac{f^{\tt found} - \bar
f}{1+|\bar f|}\right)\right]\]is the number of digits of accuracy achieved by the solver. We also analyze the ability of each solver in capturing the (known) exact $\V$-dimension, by looking at the cardinality of $A(x^{\tt found})$ as in 
\eqref{eq:almostactive}, for $x^{\tt found}$ the final point found by each solver, and computing
$v^{\tt found}=|A(x^{\tt found})|-1$.

Since {\sc maxquad} is a well-known test function for bundle methods, in Table~\ref{tab-maxquad} we report separately the measures obtained for this function, running the four solvers with two starting points.
\begin{table}[H]\centering\caption{Results for {\sc maxquad} test function, $\dim\V(\bar x)=3$.}
 \begin{tabular}{ll|c c c c }
 & &\Va  & \Vb  & \Vc  & \Vd  \\
 \hline First $x_0$ &$\begin{array}{l}\mbox{\ttRA}\\v^{\tt found}\end{array}$ & $\begin{array}{c} 5\\ 3\end{array}$ & $\begin{array}{c}  2\\ 1\end{array}$ & $\begin{array}{c}  1\\1\end{array}$ & $\begin{array}{c}  3\\3\end{array}$ \\
 \hline Second $x_0$&  $\begin{array}{l}\mbox{\ttRA}\\v^{\tt found}\end{array}$ & $\begin{array}{c} 5\\ 3\end{array}$ & $\begin{array}{c}  2\\ 0\end{array}$ & $\begin{array}{c}  1\\1\end{array}$ & $\begin{array}{c}  3\\2\end{array}$ \\\hline
\end{tabular}\label{tab-maxquad}\end{table}
We observe a very good performance of \Vd, both in terms of accuracy and $\V$-dimension, which is underestimated in the second run. Such underestimation means that \Vd\/ is taking $\mathcal{U}$-steps in a larger subspace. Of course, the price to be paid (especially with our rudimentary implementation) is computing time, which passes from a few seconds
with \Va-\Vc, to 2 minutes with \Vd.

The solver performance for the remaining 600 runs was similar. For each problem and the two random starting points, we organized the output into five groups, corresponding to increasing percentages of the $\mathcal{V}$-dimension at $\bar{x}$ with respect to $n$. Each row in Table~\ref{tab-MQAve} reports for each solver the mean value of the digits of accuracy, averaged for each group. The bottom line in Table~\ref{tab-MQAve} contains the total number of instances considered for the test and the total average values for \mbox{\ttRA}.
\begin{table}[H]\centering\caption{Average \mbox{\ttRA} for 602 ({\sc maxquad} and 300 random problems, each with 2 starting points) runs.}
 \begin{tabular}{cl|c c c c }
 $\dim\V(\bar x)$ &\# of runs&\Va  & \Vb  & \Vc  & \Vd  \\
 \hline
$\in(0\%,15\%) n$ &96&      3.99 &0.78  &    0.58  &  1.44\\
$\in[15\%,30\%) n$& 182&    4.79 &1.12 &     0.89  &  1.63\\
$\in[30\%,45\%) n$&134&     3.93 &0.91 &     0.61  &  1.05\\
$\in[45\%,60\%) n$& 106&    4.21 &0.96 &     0.62  &  1.16\\
$\in[60\%,100\%) n$ &84&    5.75 &1.36 &     1.07  &  2.15\\\hline
$\in(0\%,100\%) n$&602&4.50&   1.02&   0.76&  1.46\\\hline
\end{tabular}\label{tab-MQAve}\end{table}
As conjectured, in terms of accuracy on the optimal value, \Va\/ is far superior to all the other variants. The inexact bundle method \Vc\/ performs reasonably well, but is systematically outperformed by \Vd. An interesting feature is that, in spite of using only approximate subgradient information,  \Vd\/ achieves better function accuracy than the exact classical bundle method, \Vb. This fact confirms the interest of exploiting available structure in the bundle method, even if the information is inexact.\par Table \ref{tab:vdim} below gives another indication of the performance of \Vd\/ and \Vc\/ in predicting the dimension of the $\V$-space. Out of the 602 runs, we list the number of times  that each algorithm returned the exact $\V$-dimension, the number of times $v^{\tt found}$ was within $1,$ $2$ or $5,$ and the number of times it was more than $5$ away from the correct $\V$-dimension.
\begin{table}[H]\caption{The $\V$-dimension prediction comparison between the inexact solvers.}\begin{center}\begin{tabular}{l|ccccc}\label{tab:vdim}
&\# Exact&\# $\pm1$&\# $\pm2$&\# $\pm5$&\# $>5$\\
\hline
  {\Vc}&$161~~(27\%)$&$351~~(58\%)$&$441~~(73\%)$&$528~~(88\%)$&$74~~(12\%)$\\
\Vd&$408~~(68\%)$&$486~~(81\%)$&$513~~(85\%)$&$551~~(92\%)$&$51~~~~(8\%)$\\
\hline
\end{tabular}\end{center}\end{table}
\noindent In almost $70\%$ of the runs, \Vd\/ correctly predicted the $\V$-dimension, more than double what \Vc\/ was able to do. This is a strong indicator that \Vd\/ is able to do what it is meant to do in that respect; \Vc\/ is not meant to make this prediction, so we expect to see the results that we have.
	
In order to interpret the output graphically, we created profiles for the accuracy over the full set of 602 instances, see Figure \ref{figMQF}. In the graph, each curve represents the cumulative probability distribution $\phi_s(\theta)$ of the resource ``$f$-accuracy", measured in terms of the reciprocal of \ttRA. The use of $1/$RA as an indicator stems from the fact that usually smaller values of the abscissa $\theta$ mean better performance of the resource. As in our case higher accuracy is preferred, we invert the relation to plot the profile. In this manner, in all the profiles that follow, the solvers with the highest curves are the best ones for the given indicator of performance.
\begin{figure}\begin{center}
\includegraphics[width=0.5\linewidth]{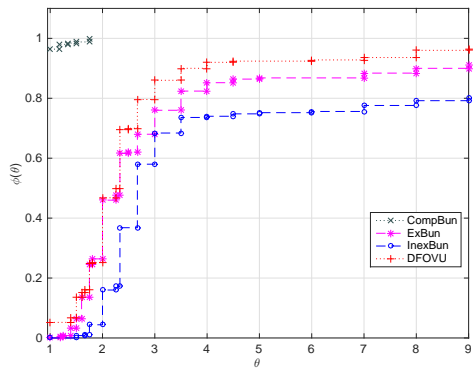}
\caption{Accuracy Profile: (reciprocal of)  accuracy, all solvers.}
\label{figMQF}\end{center}\end{figure}
\noindent For the left endpoint $\theta=\theta_{\min}$ in the graph, the probability $\phi_s(\theta_{\min})$ of a particular solver is the probability that the solver will provide the highest accuracy among all algorithms. Looking at the highest value for the left endpoint in Figure~\ref{figMQF}, we conclude that the most precise solver is \Va~ in all of the runs, as expected. The \Vd\/ solver is the second-best, followed by \Vb.\par In general, for a particular solver $s$, the ordinate $\phi_s(\theta)$ gives information on the percentage of problems that the considered method will solve if given $\theta$ times the resource employed by the best one. Looking at the value of $\theta=3$, we see that \Vd\/ solves about 85\% of the 602 problems ($\phi(3)>0.8$) with a third (=1/$\theta$) of the accuracy obtained by \Va, while \Vc\/ solves less than 70\% ($\phi(3)<0.7$).

Considering that the comparison with exact variants is not entirely fair, we repeat the profile, this time comparing only \Vc\/ and \Vd. The values of $\theta=1$ in Figure~\ref{fig-MQF34} show that \Vc\/ was more accurate than \Vd\/ in fewer than 20\% of the runs ($\phi(1)<0.2$).
\begin{figure} \begin{center}
\includegraphics[width=0.5\linewidth]{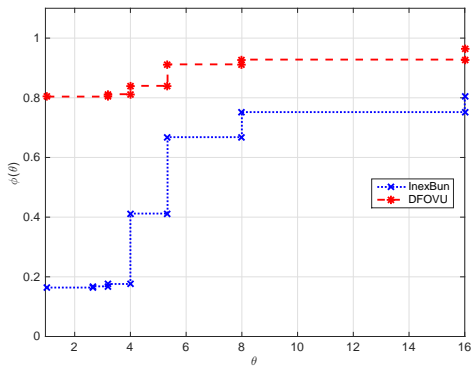}
\caption{Accuracy Profile: (reciprocal of)  accuracy, solvers \Vc\/ and \Vd.}
\label{fig-MQF34} \end{center} \end{figure}

We now comment on CPU time, function evaluations and failures of bundle solvers. Naturally, the gain in accuracy of \Vd\/ comes at the price of CPU time. As expected, the fastest solver in all of the runs is \Va, followed by \Vb, \Vc, and \Vd. The average CPU time in seconds was 0.47 for \Va, 0.28 for \Vb, 0.40 for \Vc, and 61 for \Vd. The time increase for \Vd\/ is better understood when examining the respective average number of calls to the oracle, equal to 8 for \Va, 26 for \Vb, 504 for \Vc, and 52330 for \Vd. There is a factor of close to 20 when passing from \Vb\/ to \Vc, whose only difference is in the use of the inexact (simplex) gradients. The factor of 100 between the oracle calls required by \Vc\/ and those required by \Vd\/ is explained by the fact that \Vd\/ approximates not only the gradient, but also the $\U$-Hessian. Such an increase is not a surprise, as our implementation of \Vd\/ is not optimized and the computational burden required by \Vd\/ is much higher than that required by \Vc. We comment on possible numerical enhancements in this regard in Section \ref{sec:conc}.\par Regarding failures, there was none for \Va, \Vb\/ and \Vc, whose respective stopping tests were triggered in all 602 runs. \Vd\/ failed 104 times having reached  the maximum number of allowed evaluations (\maxs), and twice when the parameter $\varepsilon_k$ became unduly small. This figure represents 17.5\% of all the runs. Most of the failures of  \Vd\/ by \maxs\/ remained even after increasing \maxs ~by a factor of 10. It is our understanding that the method reached its limit of accuracy in those instances, which likely had worse conditioning and were too difficult to solve with our inexact method. By constrast, a close observation of failures of \Vd\/ in previous runs that were due to a small $\varepsilon_k$ gave us some hints for improvement of the algorithm's performance.  We noticed that when $\varepsilon_k$ becomes too small, the stopping test in Step 1.4 becomes hard to attain and the $\mathcal{V}$-step becomes dismayingly slow. It is important to tune the manner in which $\varepsilon_k$ decreases, so that the reduction is not done too fast. For our experiments, we update $\varepsilon_k$ in Steps 3.1 and 3.3 of \Vd ~by $\varepsilon_{k+1}=0.9\varepsilon_k$. This appeared a reasonable setting for the considered 602 instances. These precautions help to ensure that the solution to \eqref{eq:ustepalg} does not become near-singular.\par
We finish by noting that \Vb, the classical bundle method, is extremely reliable, but neither as accurate nor as fast as \Va, which fully exploits structure of composite functions and uses exact gradient information. Of the four solvers, if the gradient evaluations can be done exactly, \Va\/ is to be preferred. Otherwise, \Vd\/ seems a good option for cases when accuracy of the solution is a more important concern than solving time.

\subsection{Benchmark of DFO solvers}
Having analyzed the qualities and weaknesses of \Vd\/ with respect to other bundle methods, we now examine the behaviour of \Vd\/ when compared to an established DFO solver, \Ve\/ \cite{le2009nomad}. \Ve\/ is a program that uses as its basis the MADS (Mesh Adaptive Direct Search) algorithm \cite{audet2008nonsmooth}. In the MADS algorithm, trial points on a mesh are evaluated and the mesh size for the next iteration is adjusted according to the findings of the current one. We refer the reader to \cite{audet2008nonsmooth,le2009nomad} for details on the implementation of \Ve\/ and the structure of the MADS algorithm.

In order to run \Ve\/ using Matlab, a mex-file was generated from the GERAD version of the package. The parameters were set using the built-in command
\begin{verbatim}
opts=nomadset(`display_degree',0,`bb_output_type',`obj')\end{verbatim}
Using the same battery in Table~\ref{tab-MQAve}, \Ve\/ took a very long time to trigger the stopping test for some functions. For this reason, to make the comparisons both methods
were given a time budget of 60 minutes. The \Ve\/ output corresponds to either the numerical solution found by the solver if the stopping test was triggered, or the last computed point if the maximum CPU time was attained.

Table~\ref{tab-DFO} reports the accuracy obtained by the two solvers.  The first row provides the results for the classical {\sc maxquad} problem.  The remaining rows provide averages for randomly generated problems, subdivided by dimenision. 
 
\begin{table}[H]\centering\caption{Average \mbox{\ttRA} for the DFO solvers with time budget}
 \begin{tabular}{|ll||c| c| c| c| c| c| c| c| }\hline
Problem set && \Vd  & \Ve  \\ \hline 
{\sc maxquad} && 2.17& 2.75\\ 
 \hline $n=10$ && 3.82& 1.47\\ 
 \hline $n=20$ && 2.85& 0.59\\ 
 \hline $n=30$ && 1.29& 0.32\\
 \hline $n=40$ && 1.11& 0.18\\
 \hline $n=50$ && 1.03&0.13\\
 \hline 
\end{tabular}
\label{tab-DFO}
\end{table}
For {\sc maxquad}, both solvers have similar behaviour with a slight superiority for \Ve.
Barring {\sc maxquad} runs, \Vd\/ systematically outperforms \Ve.
This is not a surprise, as \Ve\/ is a general purpose solver, designed
to solve general (constrained) problems. By contrast, \Vd\/ is
a specialized method that fully 
exploits the knowledge of the max-structure of the objective function in 
\eqref{prob1}. 

The performance profiles in Figure \ref{fig-both} illustrate the same phenomena in a graphical manner.
\begin{figure} \begin{center}
\includegraphics[width=0.9\linewidth]{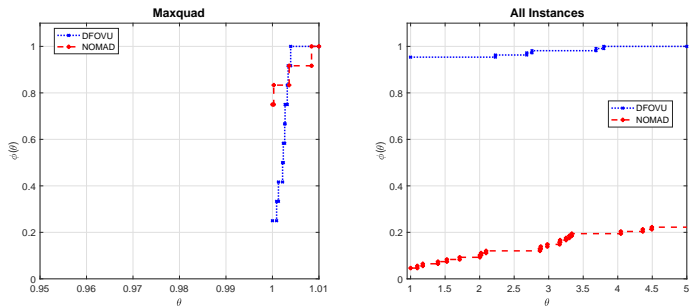}
\caption{Accuracy Profile: (reciprocal of)  accuracy, solvers \Vd\/ and \Ve.}
\label{fig-both} \end{center} \end{figure}
In the case of the {\sc maxquad} function (left graph), the performance of the two solvers is so close that the abscissa scale is on the order of hundredths. The slight initial advantage of \Ve\/ is visible; the leftmost part of the red curve lies above that of the blue curve. However, in the case of the randomly generated functions (right graph), \Vd\/ is obviously the solver of choice. We constructed profiles separated by dimension size as well, but they are all similar to the one with all dimensions included here, so the other graphs are not presented as they offer no new information.

As for CPU time, function calls and failures, we observed the following. The average CPU time for \Vd\/ over all runs was 51.5 seconds, whereas for \Ve\/ the average was 701.3 seconds. It was commonplace for \Ve\/ to use thousands of seconds of CPU time, with 35 instances reaching the imposed limit of 60 minutes. This is not unexpected, as the slower convergence speed of the MADS algorithm has been observed in previous literature \cite{bHurmen2015mesh,rabaiotti2011inverse,vasant2012hybrid}. The number of function calls is more comparable; the average was 57,857 for \Vd\/ and 41,895 for \Ve. \Ve\/ required fewer function calls than \Vd\/ fairly consistently, but the two numbers were always on the same order of magnitude. There were no failures on either side, except that \Ve\/ timed out on 35 runs. However, each of those instances returned a final function value of about 1 or 2, still reasonably close to the true minimum of 0. Given these details, we believe that the comparison is fair and represents a good performance of \Vd.

\subsection{Behaviour for nonconvex problems}

Even though our convergence analysis was developed for convex problems only, we also ran \Vd\/ on a battery of nonconvex functions, to check numerically how the method behaves in this case.  The test functions are of the form \eqref{def-testf}, again randomly generated with a known functional value at a critical point, satisfying $0\in\partial f(\bar x)$ for Clarke's subdifferential. Nonconvexity is induced by taking, among all the matrices $H_j$ defining the quadratic subfunctions in \eqref{def-testf}, at least one that is negative definite. However,
since we are dealing with unconstrained problems, to have a local solution that is finite-valued, one of the $m$ random matrices is forced to be positive definite. 

We note that only \Vd\/ was tested on these problems.  Algorithms \Vc, \Vb, and \Va, are specifically designed for convex functions and no longer work when subgradients from nonconvex functions are used.  We were unable to make \Ve\/ perform in a reasonable manner on these problems and prefer not to present suboptimal results.

This test set consists of 1000 test functions, 200 problems with 5 starting points each, with $n\in \{10, 20, 30, 40, 50\}$. The results are encouraging.  In fact, without making any changes in the implementation, the \Vd\/ stopping test was triggered in 958 cases.  For these successful runs, Table~\ref{tab-DFONC} reports the number of oracle calls, CPU time and digits of accuracy, again shown in average separately for each one of the five groups of test functions. 
 
\begin{table}[H]\centering\caption{Average output for
\Vd\/ on 958 successful nonconvex runs }
 \begin{tabular}{|l||c| c| c| }
\hline Problem set & grey-box calls  &  time&\ttRA \\
 \hline $n=10$ \scriptsize(189 successful runs) &2067& 0.87& 1.08\\ 
 \hline $n=20$ \scriptsize(200 successful runs) &3135& 0.96& 0.75\\ 
 \hline $n=30$ \scriptsize(191 successful runs) &1440& 0.80& 0.44\\
 \hline $n=40$ \scriptsize(190 successful runs) &2659& 1.67& 0.35\\
 \hline $n=50$ \scriptsize(188 successful runs) &2690& 2.00&0.16\\
 \hline 
\end{tabular}

\label{tab-DFONC}
\end{table}

Clearly, the accuracy levels are not as good as for the convex case. However, the time and grey-box calls were improved over the convex case.  This suggests that the stopping condition is somehow easier to trigger in the nonconvex setting.   In general, the output is consistent, with worse indicators for problems in higher dimensions. The group with $n=30$ is an exception, since it required fewer function evaluations (1440) than the easier instances, with $n=10$ or 20.  The fact that the solver ends up with false positive output is not unexpected, considering the stopping test is designed on the basis of the convergence analysis, which holds for convex problems only. 

Among the 42 runs in which \Vd\/ failed, the parameter $\varepsilon_k$ became too small for the problems in higher dimensions ($n\geq30$). Once more, this is not surprising, as the method had already shown high sensitivity to that parameter in the convex instances. In lower dimensions ($n\leq20$), sometimes \Vd\/ failed in building a suitable matrix $\hat{M}$. In view of Remark~\ref{rem:zero}, this suggests the need to fine-tune the parameter $\varepsilon_k$, to adapt its iterative definition to the nonconvex setting.

Overall, the results indicate that it might be worthwhile to extend both the theory and the implementation of \Vd\/ to tackle nonconvex functions.
\section{Conclusion}\label{sec:conc}

We have presented a complete and fully-functional DFO $\VU$-algorithm for convex finite-max objective functions on $\R^n$ under reasonable assumptions. This extends the original algorithm of \cite{vualg} into the derivative-free setting, where exact function values are available but approximations of subgradients are sufficient for convergence. Numerical testing suggests that, at the expense of increased CPU time and number of function calls, the DFO $\VU$-algorithm provides an improvement on final function value accuracy when compared to other inexact methods, and even compared to the \Vb ~method that uses exact first-order information. Convergence rate analysis was not performed in this paper; we leave that for a future project.\par

We point out that the numerical testing was done using a proof-of-concept implementation \Vd\/ and that there is much room for improvement of its performance. We did hand-tune the parameters to get good performance, but other tweaks in the code that were not done would help as well. For instance, a hard reset happens at every iteration, which means that nearby function values already calculated are not reused in the construction of the next model function. Retaining a cache of function calls and referencing it before making new evaluations would reduce the total number of grey-box calls. In addition, in the construction of the simplex gradient we used the coordinate directions. A method such as Householder transformation \cite{pressnumerical2007} could be used to rotate the coordinates so that the first canonical vector points in the previous descent direction. We expect these adjustments to reduce the number of function calls by a factor between $n$ and $n^2,$ so it is encouraging to know that future work on this project should result in quite a significant enhancement of the algorithm.\par 

As mentioned in the introduction, one weakness of this algorithm is that convergence applies the assumption that the objective function is convex.  It is unclear how strong an assumption this is for finite-max grey-box functions, however it would obviously be beneficial if the assumption could be relaxed.  One starting point might be the research on $\VU$-structures for nonconvex functions \cite{mifflinesagastizabal2004}.

Another interesting approach may come from the line of recent work by Salom\~ao, Santos, and Sim\~oes: 
 \cite{salomao2016differentiability,salomao2017local,salomao2016second}. In \cite{salomao2016second}, the authors present a gradient sampling method that has improved convergence speed, thanks to $\VU$-decomposition. They stress the point that gradient sampling is convenient when the objective function is nonconvex, avoiding the complications that arise when bundle methods such as the $\VU$-algorithm are applied to nonconvex functions. The algorithm therein retains some components of the $\VU$-algorithm in order to speed up convergence; it uses quasi-Newton techniques in the $\U$-space and cutting-plane techniques in the $\V$-space. It is our belief that the derivative-free methods presented in this paper should be applicable in the algorithm of \cite{salomao2016second} and other similar algorithms. Exact first-order data are employed currently, but since we have seen that such gradients (at least for finite-max problems, which the authors of \cite{salomao2016second} cite as motivation for their method) can be approximated to any desired degree of accuracy, there is good reason to conjecture that the same approach would work there. It is a natural direction for the continuation of this line of research.

\bibliographystyle{spmpsci}
\bibliography{Bibliography}{}

\def\cprime{$'$}
\begin{thebibliography}{10}
\providecommand{\url}[1]{{#1}}
\providecommand{\urlprefix}{URL }
\expandafter\ifx\csname urlstyle\endcsname\relax
  \providecommand{\doi}[1]{DOI~\discretionary{}{}{}#1}\else
  \providecommand{\doi}{DOI~\discretionary{}{}{}\begingroup
  \urlstyle{rm}\Url}\fi

\bibitem{AbramsonAudetDennis09}
Abramson, M., Audet, C., Dennis Jr., J., Le~Digabel, S.: Ortho{MADS}: a
  deterministic {MADS} instance with orthogonal directions.
\newblock SIAM J. Optim. \textbf{20}(2), 948--966 (2009).
\newblock \doi{10.1137/080716980}

\bibitem{vanackooij2014level}
Ackooij, W.v., Oliveira, W.d.: Level bundle methods for constrained convex
  optimization with various oracles.
\newblock Comput. Optim. Appl. \textbf{57}(3), 555--597 (2014).
\newblock \doi{10.1007/s10589-013-9610-3}

\bibitem{ApkarianNollRavanbod2016}
Apkarian, P., Noll, D., Ravanbod, L.: Nonsmooth bundle trust-region algorithm
  with applications to robust stability.
\newblock Set-Valued Var. Anal. \textbf{24}(1), 115--148 (2016).
\newblock \doi{10.1007/s11228-015-0352-5}

\bibitem{asplund1969gradients}
Asplund, E., Rockafellar, R.: Gradients of convex functions.
\newblock Trans. Amer. Math. Soc. \textbf{139}, 443--467 (1969).
\newblock \doi{10.2307/1995335}

\bibitem{Audet2014}
Audet, C.: A Survey on Direct Search Methods for Blackbox Optimization and
  their Applications, pp. 31--56.
\newblock Springer, New York (2014).
\newblock \doi{10.1007/978-1-4939-1124-0_2}

\bibitem{AudetBechardLeDigabel08}
Audet, C., B\'{e}chard, V., Le~Digabel, S.: Nonsmooth optimization through mesh
  adaptive direct search and variable neighborhood search.
\newblock J. Global Optim. \textbf{41}(2), 299--318 (2008).
\newblock \doi{10.1007/s10898-007-9234-1}

\bibitem{audet2008nonsmooth}
Audet, C., B{\'e}chard, V., Le~Digabel, S.: Nonsmooth optimization through mesh
  adaptive direct search and variable neighborhood search.
\newblock Journal of Global Optimization \textbf{41}(2), 299--318 (2008)

\bibitem{AudetDennis06}
Audet, C., Dennis Jr., J.: Mesh adaptive direct search algorithms for
  constrained optimization.
\newblock SIAM J. Optim. \textbf{17}(1), 188--217 (2006).
\newblock \doi{10.1137/040603371}

\bibitem{AudetHare_DFObook}
Audet, C., Hare, W.: Derivative-free and Blackbox Optimization.
\newblock Springer International Publishing AG, Switzerland (2017).
\newblock (Final edits submitted Aug., 2017. Published online Dec.\ 2017.
  Hardcopy available Jan.\ 2018.)

\bibitem{BagirovKarasSezer08}
Bagirov, A., Karas\"{o}zen, B., Sezer, M.: Discrete gradient method:
  derivative-free method for nonsmooth optimization.
\newblock J. Optim. Theory Appl. \textbf{137}(2), 317--334 (2008).
\newblock \doi{10.1007/s10957-007-9335-5}

\bibitem{BigdeliHareNutiniTesfamariam-2016}
Bigdeli, K., Hare, W., Nutini, J., Tesfamariam, S.: Optimizing damper
  connectors for adjacent buildings.
\newblock Optimization and Engineering \textbf{17}(1), 47--75 (2016).
\newblock \doi{10.1007/s11081-015-9299-5}

\bibitem{bonnans-gilbert-lemarechal-sagastizabal-2006}
Bonnans, J., Gilbert, J., Lemar{\'e}chal, C., Sagastiz{\'a}bal, C.: Numerical
  Optimization. Theoretical and Practical Aspects.
\newblock Universitext. Springer-Verlag, Berlin (2006).
\newblock Second edition, pp. xiv+490

\bibitem{bHurmen2015mesh}
B{\H{u}}rmen, {\'A}., Olen{\v{s}}ek, J., Tuma, T.: Mesh adaptive direct search
  with second directional derivative-based hessian update.
\newblock Computational Optimization and Applications \textbf{62}(3), 693--715
  (2015)

\bibitem{proxsplit}
Combettes, P., Pesquet, J.: Proximal splitting methods in signal processing.
\newblock In: Fixed-point algorithms for inverse problems in science and
  engineering, \emph{Springer Optim. Appl.}, vol.~49, pp. 185--212. Springer,
  New York (2011).
\newblock \doi{10.1007/978-1-4419-9569-8_10}

\bibitem{ConnScheinbergToint1997}
Conn, A., Scheinberg, K., Toint, P.: On the convergence of derivative-free
  methods for unconstrained optimization.
\newblock In: Approximation theory and optimization ({C}ambridge, 1996), pp.
  83--108. Cambridge Univ. Press, Cambridge (1997)

\bibitem{connscheinbergvicente_book}
Conn, A., Scheinberg, K., Vicente, L.: Introduction to derivative-free
  optimization, \emph{MPS/SIAM Series on Optimization}, vol.~8.
\newblock Society for Industrial and Applied Mathematics (SIAM), Philadelphia,
  PA; Mathematical Programming Society (MPS), Philadelphia, PA (2009).
\newblock \doi{10.1137/1.9780898718768}

\bibitem{correaconvergence1993}
Correa, R., Lemar\'echal, C.: Convergence of some algorithms for convex
  minimization.
\newblock Math. Program. \textbf{62}(2, Ser. B), 261--275 (1993).
\newblock \doi{10.1007/BF01585170}

\bibitem{CustVicente2007}
Cust{\'o}dio, A., Vicente, L.: Using sampling and simplex derivatives in
  pattern search methods.
\newblock SIAM J. Optim. \textbf{18}(2), 537--555 (2007).
\newblock \doi{10.1137/050646706}

\bibitem{FrangioniGorgone2014}
Frangioni, A., Gorgone, E.: Bundle methods for sum-functions with ``easy''
  components: applications to multicommodity network design.
\newblock Math. Program. \textbf{145}(1-2, Ser. A), 133--161 (2014).
\newblock \doi{10.1007/s10107-013-0642-3}

\bibitem{FuduliGaudiosoGial2004}
Fuduli, A., Gaudioso, M., Giallombardo, G.: A {DC} piecewise affine model and a
  bundling technique in nonconvex nonsmooth minimization.
\newblock Optim. Methods Softw. \textbf{19}(1), 89--102 (2004).
\newblock \doi{10.1080/10556780410001648112}

\bibitem{numvu}
Hare, W.: Numerical analysis of {$VU$}-decomposition, {$U$}-gradient, and
  {$U$}-hessian approximations.
\newblock SIAM J. Optim. \textbf{24}(4), 1890--1913 (2014).
\newblock \doi{10.1137/130933691}

\bibitem{HareLucet2014}
Hare, W., Lucet, Y.: Derivative-free optimization via proximal point methods.
\newblock J. Optim. Theory Appl. \textbf{160}(1), 204--220 (2014).
\newblock \doi{10.1007/s10957-013-0354-0}

\bibitem{Hare-Macklem-2013}
Hare, W., Macklem, M.: Derivative-free optimization methods for finite minimax
  problems.
\newblock Optim. Methods Softw. \textbf{28}(2), 300--312 (2013).
\newblock \doi{10.1080/10556788.2011.638923}

\bibitem{Hare-Nutini-2013}
Hare, W., Nutini, J.: A derivative-free approximate gradient sampling algorithm
  for finite minimax problems.
\newblock Comput. Optim. Appl. \textbf{56}(1), 1--38 (2013).
\newblock \doi{10.1007/s10589-013-9547-6}

\bibitem{HareNutiniTesfamariam-2013}
Hare, W., Nutini, J., Tesfamariam, S.: A survey of non-gradient optimization
  methods in structural engineering.
\newblock Adv. Eng. Soft. \textbf{59}, 19--28 (2013).
\newblock \doi{10.1016/j.advengsoft.2013.03.001}

\bibitem{hareplaniden2016}
Hare, W., Planiden, C.: {C}omputing proximal points of convex functions with
  inexact subgradients.
\newblock (to appear) Set-Valued Var. Anal. pp. 1--24 (~)

\bibitem{HareSagaSolodov2016}
Hare, W., Sagastiz\'abal, C., Solodov, M.: A proximal bundle method for
  nonsmooth nonconvex functions with inexact information.
\newblock Comput. Optim. Appl. \textbf{63}(1), 1--28 (2016).
\newblock \doi{10.1007/s10589-015-9762-4}

\bibitem{HelmbergOvertonRendl2014}
Helmberg, C., Overton, M., Rendl, F.: The spectral bundle method with
  second-order information.
\newblock Optim. Methods Softw. \textbf{29}(4), 855--876 (2014).
\newblock \doi{10.1080/10556788.2013.858155}

\bibitem{HelmbergRendl2000}
Helmberg, C., Rendl, F.: A spectral bundle method for semidefinite programming.
\newblock SIAM J. Optim. \textbf{10}(3), 673--696 (2000).
\newblock \doi{10.1137/S1052623497328987}

\bibitem{hiriart1993convex2}
Hiriart-Urruty, J.B., Lemar\'echal, C.: Convex analysis and minimization
  algorithms. {II}, \emph{Grundlehren der Mathematischen Wissenschaften
  [Fundamental Principles of Mathematical Sciences]}, vol. 306.
\newblock Springer-Verlag, Berlin (1993).
\newblock Advanced theory and bundle methods

\bibitem{JokiBagirovKarmitsa2017}
Joki, K., Bagirov, A., Karmitsa, N., M\"akel\"a, M.: A proximal bundle method
  for nonsmooth {DC} optimization utilizing nonconvex cutting planes.
\newblock J. Global Optim. \textbf{68}(3), 501--535 (2017).
\newblock \doi{10.1007/s10898-016-0488-3}

\bibitem{KarmitsaBagirovTaheri2017}
Karmitsa, N., Bagirov, A., Taheri, S.: New diagonal bundle method for
  clustering problems in large data sets.
\newblock European J. Oper. Res. \textbf{263}(2), 367--379 (2017).
\newblock \doi{10.1016/j.ejor.2017.06.010}

\bibitem{Kelley99}
Kelley, C.: Iterative methods for optimization, \emph{Frontiers in Applied
  Mathematics}, vol.~18.
\newblock Society for Industrial and Applied Mathematics (SIAM), Philadelphia,
  PA (1999).
\newblock \doi{10.1137/1.9781611970920}

\bibitem{Kelley11}
Kelley, C.: Implicit filtering, \emph{Software, Environments, and Tools},
  vol.~23.
\newblock Society for Industrial and Applied Mathematics (SIAM), Philadelphia,
  PA (2011).
\newblock \doi{10.1137/1.9781611971903}

\bibitem{kiwiel1990proximity}
Kiwiel, K.: Proximity control in bundle methods for convex nondifferentiable
  minimization.
\newblock Math. Program. \textbf{46}(1, (Ser. A)), 105--122 (1990).
\newblock \doi{10.1007/BF01585731}

\bibitem{Kiwiel2006}
Kiwiel, K.: A proximal bundle method with approximate subgradient
  linearizations.
\newblock SIAM J. Optim. \textbf{16}(4), 1007--1023 (2006).
\newblock \doi{10.1137/040603929}

\bibitem{Kiwiel10}
Kiwiel, K.: A nonderivative version of the gradient sampling algorithm for
  nonsmooth nonconvex optimization.
\newblock SIAM J. Optim. \textbf{20}(4), 1983--1994 (2010).
\newblock \doi{10.1137/090748408}

\bibitem{LarsonMenickellyWild2016}
Larson, J., Menickelly, M., Wild, S.: Manifold sampling for $\ell_1$ nonconvex
  optimization.
\newblock SIAM J. Optim. \textbf{26}(4), 2540--2563 (2016).
\newblock \doi{10.1137/15M1042097}

\bibitem{le2009nomad}
Le~Digabel, S.: Nomad: Nonlinear optimization with the mads algorithm.
\newblock Rapport technique G-2009-39, Les cahiers du GERAD  (2009)

\bibitem{ulagconv}
Lemar{\'e}chal, C., Oustry, F., Sagastiz{\'a}bal, C.: The {$U$}-{L}agrangian of
  a convex function.
\newblock Trans. Amer. Math. Soc. \textbf{352}(2), 711--729 (2000).
\newblock \doi{10.1090/S0002-9947-99-02243-6}

\bibitem{LewisWright2016}
Lewis, A., Wright, S.: A proximal method for composite minimization.
\newblock Math. Program. \textbf{158}(1-2, Ser. A), 501--546 (2016).
\newblock \doi{10.1007/s10107-015-0943-9}

\bibitem{vudecomp}
Mifflin, R., Sagastiz{\'a}bal, C.: {$VU$}-decomposition derivatives for convex
  max-functions.
\newblock In: Ill-posed variational problems and regularization techniques
  ({T}rier, 1998), \emph{Lecture Notes in Econom. and Math. Systems}, vol. 477,
  pp. 167--186. Springer, Berlin (1999).
\newblock \doi{10.1007/978-3-642-45780-7_11}

\bibitem{mifflin2000functions}
Mifflin, R., Sagastiz\'abal, C.: Functions with primal-dual gradient structure
  and {$U$}-{H}essians.
\newblock In: Nonlinear optimization and related topics ({E}rice, 1998),
  \emph{Appl. Optim.}, vol.~36, pp. 219--233. Kluwer Acad. Publ., Dordrecht
  (2000).
\newblock \doi{10.1007/978-1-4757-3226-9_12}

\bibitem{vutheory}
Mifflin, R., Sagastiz{\'a}bal, C.: On {$VU$}-theory for functions with
  primal-dual gradient structure.
\newblock SIAM J. Optim. \textbf{11}(2), 547--571 (2000).
\newblock \doi{10.1137/S1052623499350967}

\bibitem{mifflin-sagastizabal-2002}
Mifflin, R., Sagastiz{\'a}bal, C.: Proximal points are on the fast track.
\newblock J. Convex Anal. \textbf{9}(2), 563--579 (2002).
\newblock
  \urlprefix\url{http://www.heldermann.de/JCA/JCA09/JCA092/jca09033.htm}

\bibitem{mifflin2003primal}
Mifflin, R., Sagastiz\'abal, C.: Primal-dual gradient structured functions:
  second-order results; links to epi-derivatives and partly smooth functions.
\newblock SIAM J. Optim. \textbf{13}(4), 1174--1194 (2003).
\newblock \doi{10.1137/S1052623402412441}

\bibitem{vusmoothness}
Mifflin, R., Sagastiz{\'a}bal, C.: {$VU$}-smoothness and proximal point results
  for some nonconvex functions.
\newblock Optim. Methods Softw. \textbf{19}(5), 463--478 (2004).
\newblock \doi{10.1080/10556780410001704902}

\bibitem{mifflinesagastizabal2004}
Mifflin, R., Sagastiz\'{a}bal, C.: {$VU$}-smoothness and proximal point results
  for some nonconvex functions.
\newblock Optim. Methods Softw. \textbf{19}(5), 463--478 (2004).
\newblock \doi{10.1080/10556780410001704902}

\bibitem{vualg}
Mifflin, R., Sagastiz{\'a}bal, C.: A {$VU$}-algorithm for convex minimization.
\newblock Math. Program. \textbf{104}(2-3, Ser. B), 583--608 (2005).
\newblock \doi{10.1007/s10107-005-0630-3}

\bibitem{nel2016continuity}
Nel, L.: Continuity Theory, first edn.
\newblock Springer (2016).
\newblock \doi{10.1007/978-3-319-31159-3}

\bibitem{NollProtRondepierre2008}
Noll, D., Prot, O., Rondepierre, A.: A proximity control algorithm to minimize
  nonsmooth and nonconvex functions.
\newblock Pac. J. Optim. \textbf{4}(3), 571--604 (2008)

\bibitem{deOliveira2017}
Oliveira, W.d.: Proximal bundle methods for nonsmooth {DC} programming.
\newblock preprint  (2017).
\newblock \url{http://www.oliveira.mat.br/publications}

\bibitem{deOliveiraSagaLemar2014}
Oliveira, W.d., Sagastiz\'abal, C., Lemar\'echal, C.: Convex proximal bundle
  methods in depth: a unified analysis for inexact oracles.
\newblock Math. Program. \textbf{148}(1-2, Ser. B), 241--277 (2014).
\newblock \doi{10.1007/s10107-014-0809-6}

\bibitem{deOliveiraSolodov2016}
Oliveira, W.d., Solodov, M.: A doubly stabilized bundle method for nonsmooth
  convex optimization.
\newblock Math. Program. \textbf{156}(1-2, Ser. A), 125--159 (2016).
\newblock \doi{10.1007/s10107-015-0873-6}

\bibitem{Powell2008}
Powell, M.: Developments of {NEWUOA} for minimization without derivatives.
\newblock IMA J. Numer. Anal. \textbf{28}(4), 649--664 (2008).
\newblock \doi{10.1093/imanum/drm047}

\bibitem{pressnumerical2007}
Press, W., Teukolsky, S., Vetterling, W., Flannery, B.: Numerical Recipes: The
  Art of Scientific Computing, third edn.
\newblock Cambridge University Press, Cambridge (2007)

\bibitem{rabaiotti2011inverse}
Rabaiotti, C.: Inverse analysis in road geotechnics, vol. 18135.
\newblock vdf Hochschulverlag AG (2011)

\bibitem{rockwets}
Rockafellar, R., Wets, J.B.: Variational analysis.
\newblock Grundlehren der Mathematischen Wissenschaften [Fundamental Principles
  of Mathematical Sciences]. Springer-Verlag, Berlin (1998).
\newblock \doi{10.1007/978-3-642-02431-3}

\bibitem{sagastizabal-2013}
Sagastiz{\'a}bal, C.: { Composite proximal bundle method}.
\newblock Math. Program. \textbf{140}(1), 189--233 (2013).
\newblock \doi{10.1007/s10107-012-0600-5}

\bibitem{salomao2016differentiability}
Salom\~ao, E., Santos, S., Sim\~oes, L.: On the differentiability check in
  gradient sampling methods.
\newblock Optim. Methods Softw. \textbf{31}(5), 983--1007 (2016).
\newblock \doi{10.1080/10556788.2016.1178262}

\bibitem{salomao2017local}
Salom\~ao, E., Santos, S., Sim\~oes, L.: On the local convergence analysis of
  the gradient sampling method.
\newblock preprint, Optimization Online  (2016).
\newblock \urlprefix\url{www.optimization-online.org/DB_HTML/2016/10/5683.html}

\bibitem{salomao2016second}
Salom\~ao, E., Santos, S., Sim\~oes, L.: A second-order information-based
  gradient and function sampling method for nonconvex nonsmooth optimization.
\newblock preprint, Optimization Online  (2017).
\newblock \urlprefix\url{www.optimization-online.org/DB_HTML/2016/06/5513.html}

\bibitem{vasant2012hybrid}
Vasant, P.: Hybrid mesh adaptive direct search and genetic algorithms for
  solving fuzzy non-linear optimization problems.
\newblock In: ICT and Knowledge Engineering (ICT \& Knowledge Engineering),
  2011 9th International Conference on, pp. 88--93. IEEE (2012)

\end{thebibliography}

\end{document}